\documentclass{article}

\usepackage{hyperref}
\usepackage[totalheight=21 true cm, totalwidth=15.2 true cm]{geometry}

\usepackage{xypic}

\usepackage{enumerate}
\usepackage{amsthm}
\usepackage{amsmath}
\usepackage{amssymb}
\usepackage{amsfonts}
\usepackage{dsfont}

\newtheorem{theo}{Theorem}[section]
\newtheorem{lemma}[theo]{Lemma}
\newtheorem{prop}[theo]{Proposition}
\newtheorem{cor}[theo]{Corollary}
\newtheorem{defi}[theo]{Definition}

\theoremstyle{definition}
\newtheorem{ex}[theo]{Example}

\newtheorem{rem}[theo]{Remark}

 \renewcommand{\labelenumi}{{\rm (\roman{enumi})}}
 
\newcommand{\f}{\phi}

\newcommand{\spec}{\operatorname{Spec}}

\newcommand{\Gl}{\operatorname{GL}}

\newcommand{\Aut}{\operatorname{Aut}}
\newcommand{\Hom}{\operatorname{Hom}}
\newcommand{\im}{\operatorname{Im}}

\newcommand{\id}{\operatorname{id}}

\newcommand{\de}{\delta}

\newcommand{\I}{\mathbb{I}}

\newcommand{\Gm}{\mathbb{G}_m}

\newcommand{\Rep}{\operatorname{Rep}}
\newcommand{\M}{\mathcal{M}}

\renewcommand{\vec}{\operatorname{Vec}}
\newcommand{\autt}{\underline{\operatorname{Aut}}^\otimes}
\newcommand{\PRO}{\operatorname{PROALGEBRAIC\ \ GROUPS}}
\newcommand{\TANN}{\operatorname{TANN}}
\newcommand{\PROALG}{\operatorname{PROALG}}

\newcommand{\tp}{\operatorname{tp}}

\newcommand{\Th}{\operatorname{Th}}

\title{Model Theory of  Proalgebraic Groups
\author{Anand Pillay\thanks{The first author was supported by the NSF grants DMS-1360702, DMS-1665035 and DMS-1760212.} \   and Michael Wibmer \thanks{The second author was supported by the NSF grants DMS-1760212, DMS-1760413, DMS-1760448 and the Lise Meitner grant M-2582-N32 of the Austrian Science Fund FWF.}}}

\date{\today}

\begin{document}
\maketitle

\begin{abstract}	
	\let\thefootnote\relax
We lay the foundations for a model theoretic study of proalgebraic groups. Our axiomatization is based on the tannakian philosophy. Through a tensor analog of skeletal categories we are able to consider neutral tannakian categories with a fibre functor as many-sorted first order structures. The class of diagonalizable proalgebraic groups is analyzed in detail. We show that the theory of a diagonalizable proalgebraic group $G$ is determined by the theory of the base field and the theory of the character group of $G$. Some initial steps towards a comprehensive study of types are also made.   
\end{abstract}

\section*{Introduction}

Our initial inspiration for this paper goes back to the model theoretic treatment of profinite groups developed by G. Cherlin, L. van den Dries, A. Macintyre and Z. Chatzidakis in the eighties.  (See \cite{CherlinVanDenDriesMacintyre:DecidabilityAndUndecidabilityTheoremsforPACfields}, \cite{CherlinVanDenDriesMacinyre},\cite{Chatzidakis:ModelTheoryOfProfiniteGroupsPhDThesis}, \cite{Chatzidakis:ModelTheoryOfProfiniteGroupsHavingTheIwasawProperty},\cite{Chatzidakis:ModelTheoryOfProfiniteGroupsHavingIPIII}.) To a profinite group $G$ they associate an $\omega$-sorted structure consisting of the cosets $gN$ of the open normal subgroups $N$ of $G$; the coset $gN$ is of sort $n$ if $[G:N]\leq n$. In an appropriate language these structures can be axiomatized by a theory $T$ and there is an anti-equivalence of categories between the category of profinite groups with epimorphisms as morphisms and the category of models of $T$ with embeddings as morphisms.
A certain extension $T_{IP}$ of $T$ is particularly well-behaved. The theory $T_{IP}$ axiomatizes profinite groups $G$ having the Iwasawa (or embedding) property: Any diagram
$$
\xymatrix{
G \ar@{..>}[d] \ar[rd]  & \\
B \ar[r] & A	
}
$$
where $B\to A$ is an epimorphism of finite groups and $G\to A$ is an epimorphism can be completed to a commutative diagram via an epimorphism $G\to B$, if $B$ is a quotient of $G$.
The theory of a profinite group having the Iwasawa property is $\omega$-categorical and $\omega$-stable. Moreover, the saturated models of $T_{IP}$ are exactly the free profinite groups.

Some parts of the theory of free profinite groups have recently been generalized to proalgebraic groups (\cite{Wibmer:FreeProalgebraicGroups}). This begs the questions, which aspects, if any, of the model theory of profinite groups have a proalgebraic counterpart? To begin with, it is a priori rather unclear how to treat proalgebraic groups as first-order structures. One may envision that the role played by the cardinalities $|G/N|$ of the finite quotients of a profinite group $G$ could be replaced by the degrees of defining equations of the algebraic quotients $G/N$ of a proalgebraic group $G$. However, a key fact used in the axiomatization of profinite groups is that if $N_1$ and $N_2$ are open normal subgroups of a profinite group $G$, then $|G/(N_1\cap N_2)|$ is bounded by $|G/N_1|\cdot |G/N_2|$. The degree does not exhibit such a behavior.

The main achievement of this paper is the introduction of a many-sorted language that allows us to axiomatize proalgebraic groups. The key idea is based on the tannakian philosophy. Instead of axiomatizing proalgebraic groups directly, we axiomatize their categories of representations, i.e., we axiomatize neutral tannakian categories together with a fibre functor. To implement this approach, certain technical challenges need to be overcome. For example, one cannot directly consider the class of all (finite dimensional, linear) representations of a proalgebraic group as a first order structure because this class is too big. Besides the fact that it is a proper class (i.e., not a set) the cardinality of the first-order structure associated to a proalgebraic group should be something algebraically meaningful, like the rank of the profinite group in the profinite setting. Therefore, one has to consider representations up to isomorphism. In other words, one has to consider skeletons of the category of representations of a proalgebraic group. To account for the fact that such a skeleton need not be closed under the tensor product, we introduce a tensor analog of skeletal categories; a notion that we deem of independent interest in the study of tensor categories.

We introduce a first-order theory $\PROALG$ in an appropriate many-sorted language such that the category of models of $\PROALG$ with the homomorphisms as morphisms is equivalent to the category of triples $(k,C,\omega)$, where $k$ is a field, $C$ a neutral tannakian category over $k$ that satisfies a tensor analog of being skeletal and $\omega$ is a fibre functor on $C$.  We also show that the functor $(k,C,\omega)\rightsquigarrow (k,\autt(\omega))$ to the category of proalgebraic groups (over varying base fields) is full, essentially surjective and induces a bijection on isomorphism classes. Thus, to every proalgebraic group $G$, there is associated a model $\M$ of $\PROALG$ that is unique up to an isomorphism. We can therefore unambiguously define the theory of $G$ as the theory of $\M$.  

Even for algebraic groups as innocuous as the multiplicative group $\Gm$ it is a non-trivial matter to determine their theory. We show that the theory of the multiplicative group over a field $k$ is determined by the theory of $k$ and the theory of its character group, i.e., by the theory of $(\mathbb{Z},+)$.
Indeed, we establish a similar result for any diagonalizable proalgebraic group. If $G$ is a proalgebraic group corresponding to a model $\M$ of $\PROALG$, then the character group of $G$ is interpretable in $\M$. If $G$ is diagonalizable with character group $A$ there is a converse: The structure $\M$ is interpretable in the structure $(k,A)$, with the language of fields on $k$ and the language of abelian groups on $A$. In fact, we show that the theory of all diagonalizable proalgebraic groups is weakly bi-interpretable with the theory of pairs $(k,A)$, where $k$ is a field and $A$ an abelian group.
 
We consider this article to be the first step in a model theoretic treatment of proalgebraic groups. Many, even rather basic questions remain open. However, we do give a flavor of the expressive power of our theory $\PROALG$ by unveiling some of the algebraic information encoded in certain types.

\medskip

There is some thematic overlap between our work and work of M. Kamensky (\cite{Kamensky:ModelTheoryAndTheTannakianFormalism}) in the sense that both articles connect model theory and tannakian categories. However, the approaches and the aims differ: Our theory $\PROALG$ axiomatizes neutral tannakian categories together with a fibre functor for the purpose of advancing the theory of proalgebraic groups using model theoretic techniques. M. Kamensky's theory $T_{\mathcal{C}}$ (in a language $\mathcal{L}_{\mathcal{C}}$ dependent on $\mathcal{C}$) axiomatizes fibre functors on a \emph{fixed} neutral tannakian category $\mathcal{C}$ for the purpose of reproving the main tannaka reconstruction theorem using model theoretic techniques. On the other hand, we feel that this article may be seen as a possible answer to the open question 0.1.2 in \cite{Kamensky:ModelTheoryAndTheTannakianFormalism}.

One of the main motivations for the model theoretic treatment of profinite groups is that it has applications in the model theory of fields, in particular the model theory of pseudo algebraically closed fields. This is based on the fact that for a field $K$, the first-order structure corresponding to the absolute Galois group of $K$ is interpretable in the field $K$. For a differential field $(K,\de)$ of characteristic zero with algebraically closed constants the absolute differential Galois group (\cite{SingerPut:differential},\cite{BachmayrHarbaterHartmannWibmer:FreeDifferentialGaloisGroups}) is a proalgebraic group. It appears that at least some reduct of the structure corresponding to the absolute differential Galois group of $(K,\de)$ is interpretable in the differential field $(K,\de)$. We therefore hope that our model theoretic treatment of proalgebraic groups will eventually lead to applications in the model theory of differential fields.

Typically model theorists treat algebraic and proalgebraic groups simply as definable respectively prodefinable groups in $\operatorname{ACF}$, the theory of algebraically closed fields. Our approach allows us to treat proalgebraic groups as structures in their own right. One advantage of our approach is that we can handle non-reduced algebraic groups, such as, e.g., the group of $p$-th roots of unity in characteristic $p$, without difficulties, whereas the point-set approach dictated by ACF is oblivious to these groups.

\medskip

The article is organized as follows: The first section is purely algebraic, i.e., does not involve any model theory. After recalling the basic definitions and results from the tannakian theory we introduce tensor skeletal tensor categories and the closely related notion of pointed skeletal neutral tannakian categories. We then proceed to define the category $\TANN$. This category has as objects tripes $(k,C,\omega)$, where $k$ is a field, $C$ a pointed skeletal neutral tannakian category over $k$ and $\omega$ a neutral fibre functor on $C$. We show that the functor $(k,C,\omega)\rightsquigarrow (k,\autt(\omega))$ from the category $\TANN$ to the category of proalgebraic groups is full, essentially surjective and induces a bijection on the isomorphism classes.

In the second section we present the axioms for $\PROALG$ in an appropriate many-sorted language. We show that the category of models of $\PROALG$ is equivalent to the category $\TANN$ and we briefly discuss some elementary classes of proalgebraic groups.

In the third section we study the theory of diagonalizable proalgebraic groups. We show that it is weakly bi-interpretable with the theory of pairs $(k,A)$, where $k$ is a field and $A$ an abelian group. From this we deduce rather directly a description of the completions of the theory of diagonalizable proalgebraic groups and a characterization of elementary extensions. It also follows that the theory of a diagonalizable proalgebraic group over an algebraically closed field is stable.

In the final section we present some initial results concerning types. The main result is that if a representation of a proalgebraic group is considered as an element of a model of $\PROALG$, then its type over the base field determines the image of the representation.

\section{Tannakian categories}

In this section we first recall the main definitions and results from the theory of tannakian categories. Then we introduce a tensor version of skeletal categories and show that the isomorphism classes of pointed skeletal neutral tannakian categories with a fibre functor are in bijection with the isomorphism classes of proalgebraic groups.

\medskip 

\medskip

\noindent {\bf \large Notation and Conventions:}

\medskip

All rings are assumed to be commutative and unital. Throughout the article $k$ denotes an arbitrary field, usually our ``base field''. The category of finite dimensional $k$-vector spaces is denoted by $\vec_k$. A \emph{proalgebraic group}\footnote{It would admittedly be more accurate to use the term ``pro-affine algebraic group'' or ``affine group scheme'' instead of ``proalgebraic group''. We hope the reader does not object to our choice of brevity over rigor in this instance.} over $k$ is, by definition, an affine group scheme over $k$. An \emph{algebraic group} over $k$ is an affine group scheme of finite type over $k$.
We will often think of a proalgebraic group $G$ as the functor $R\rightsquigarrow G(R)$ from the category of $k$-algebras to the category of groups. Conversely, a functor from the category of $k$-algebras to the category of groups is a proalgebraic group if and only if it is representable. A \emph{closed subgroup} of a proalgebraic group is a closed subgroup scheme. Some helpful references for the theory of algebraic and proalgebraic groups are \cite{Waterhouse:IntroductiontoAffineGroupSchemes}, \cite{DemazureGabriel:GroupesAlgebriques} and \cite{Milne:AlgebraicGroupsTheTheoryOfGroupSchemesOfFiniteTypeOverAField}.

For a vector space $V$ over $k$ we denote by $\Gl_V$ the functor from the category of $k$-algebras to the category of groups that assigns to any $k$-algebra $R$, the group of $R$-linear automorphisms of $V\otimes_k R$. 
A \emph{representation} of a proalgebraic group $G$ is a pair $(V,\f)$, where $V$ is a $k$-vector space and $\f\colon G\to\Gl_V$ is a morphism of functors (also called a natural transformation), i.e., $G(R)$ acts, functorially in $R$, on $V\otimes_k R$ through $R$-linear automorphisms. A morphism $f\colon (V,\f)\to (V',\f')$ of representations of $G$ is a $k$-linear map $f\colon V\to V'$ such that
$$
\xymatrix{
	V\otimes_k R \ar^-{f\otimes R}[r] \ar_{\f(g)}[d] & V'\otimes_k R \ar^{\f'(g)}[d] \\
	V\otimes_k R \ar^{f\otimes R}[r] & V'\otimes_k R	
}
$$ 
commutes for all $k$-algebras $R$ and $g\in G(R)$. All representations are assumed to be finite dimensional unless the contrary is explicitly allowed. We denote the category of all finite dimensional representations of $G$ by $\Rep(G)$.

\subsection{Recollection}
We start by recalling the basic definitions and results from the theory of tannakian categories. See \cite{Deligne:categoriestannakien} and \cite{DeligneMilne:TannakianCategories} for more details.

\begin{defi}
	A \emph{tensor category} is a tuple $(C,\otimes,\Phi,\Psi)$, where $C$ is a category, $\otimes\colon C\times C\to C$ is a functor and 
	$\Phi$ and $\Psi$ are isomorphisms of functors, called the \emph{associativity} and \emph{commutativity constraints} respectively.
	More specifically, $\Phi$ has components $\Phi_{X,Y,Z}\colon X\otimes (Y\otimes Z)\to (X\otimes Y)\otimes Z$ and $\Psi$ has components $\Psi_{X,Y}\colon X\otimes Y\to Y\otimes X$ for objects $X,Y,Z$ of $C$ such that three commutative diagrams are satisfies. See \cite[Section 1]{DeligneMilne:TannakianCategories} for details. It is also required that there exists an \emph{identity object} $(\mathds{1},u)$. This means that $C\to C,\ X\rightsquigarrow\mathds{1}\otimes X$ is an equivalence of categories and $u\colon \mathds{1}\to \mathds{1}\otimes\mathds{1}$ is an isomorphism.
\end{defi}
We will often omit $\Phi$ and $\Psi$ from the notation and refer to $C$ or the pair $(C,\otimes)$ as a tensor category. If $(\mathds{1},u)$ and $(\mathds{1}',u')$ are identity objects of a tensor category $(C,\otimes)$, then there exists a unique isomorphism $a\colon \mathds{1}\to \mathds{1}'$ such that
$$
\xymatrix{
\mathds{1} \ar_-a[d] \ar^-u[r] & \mathds{1}\otimes\mathds{1} \ar^-{a\otimes a}[d] \\
\mathds{1}' \ar^-{u'}[r] & \mathds{1}'\otimes\mathds{1}'	
}
$$
commutes (\cite[Prop. 1.3, b)]{DeligneMilne:TannakianCategories}). Moreover, there exists a unique isomorphism of functors $l$ with components $l_X\colon X\to \mathds{1}\otimes X$ such that certain diagrams commute (\cite[Prop. 1.3, a)]{DeligneMilne:TannakianCategories}. Similarly, for $r_X\colon X\to X\otimes \mathds{1}$.

\begin{defi}
	A \emph{tensor functor} from a tensor category  $(C,\otimes)$ to a tensor category $(C',\otimes')$ is a pair $(T,c)$, where $T\colon C\to C'$ is a functor and $c$ is an isomorphism of functors with components $c_{X,Y}\colon T(X)\otimes'T(Y)\to T(X\otimes Y)$ for objects $X,Y$ of $C$ such that
	\begin{enumerate}
		\item $$
		\xymatrix{
		T(X)\otimes' (T(Y)\otimes' T(Z)) \ar^{\id\otimes' c}[r] \ar_{\Phi'}[d] & 	T(X)\otimes' (T(Y\otimes T(Z))\ar^-c[r] & T(X\otimes(Y\otimes Z)) \ar^{T(\Phi)}[d] \\
		(T(X)\otimes' T(Y))\otimes'T(Z) \ar[r]^-{c\otimes'\id} & T(X\otimes Y)\otimes'T(Z) \ar^c[r] & T((X\otimes Y)\otimes Z)
		}
		$$
		commutes for all objects $X,Y,Z$ of $C$,
		\item $$
		\xymatrix{
			T(X)\otimes' T(Y) \ar_{\Psi'}[d] \ar^c[r] & T(X\otimes Y) \ar^{T(\Psi)}[d] \\
			T(Y)\otimes' T(X) \ar^c[r] & T(Y\otimes X) 
		}
		$$
		commutes for all objects $X$ and $Y$ of $C$ and
		\item if $(\mathds{1},u)$ is an identity object of $(C,\otimes)$, then $(T(\mathds{1}),T(u))$ is an identity object of $(C',\otimes')$.
	\end{enumerate}
	A tensor functor $(T,c)$ is \emph{strict} if $c$ is the identity transformation. In particular, $T(X)\otimes' T(Y)=T(X\otimes Y)$ for all objects $X$ and $Y$ of $C$.
	A tensor functor $(T,c)$ is a \emph{tensor equivalence} if $T$ is an equivalence of categories.	
\end{defi}

We will usually omit the isomorphism of functors $c$ from the notation and refer to $T$ as a tensor functor. In Section \ref{section: Neutral tannakian categories as first order structures} we will define first order structures corresponding to neutral tannakian categories. Homomorphisms between such structures correspond to strict tensor functors. Therefore we are mostly interested in strict tensor functors. We note that for $c=\id$, the above two diagrams reduce to $T(\Phi_{X,Y,Z})=\Phi'_{T(X),T(Y),T(Z)}$ and $T(\Psi_{X,Y})=\Psi'_{T(X),T(Y)}$. 

\begin{defi} \label{defi: morphism of tensor functors}
	Let $(T_1,c_1)$ and $(T_2,c_2)$ be tensor functors from $(C,\otimes)$ to  $(C',\otimes')$. A morphism $\alpha\colon T_1\to T_2$ of functors is a \emph{morphism of tensor functors} if 
	\begin{enumerate}
		\item 	$$
		\xymatrix{
			T_1(X)\otimes'T_1(Y) \ar^-{c_1}[r] \ar_{\alpha_X\otimes'\alpha_Y}[d] & T_1(X\otimes Y) \ar^{\alpha_{X\otimes Y}}[d] \\
			T_2(X)\otimes'T_2(Y) \ar^-{c_2}[r] & T_2(X\otimes Y)	
		}
		$$
		commutes for all objects $X,Y$ of $C$ and
\item 
	 $$
	 \xymatrix{
	 		& \mathds{1}' \ar[ld] \ar[rd] & \\
	 T_1(\mathds{1}) \ar^-{\alpha _{\mathds{1}}}[rr] & &T_2(\mathds{1})
 	 }
	 $$ commutes, where the downwards morphisms are deduced from the uniqueness of the identity object in $(C',\otimes')$.
	 	\end{enumerate}
\end{defi}

Note that, assuming (i), condition (ii) is equivalent to $\alpha_{\mathds{1}}$ being an isomorphism. In particular, if $\alpha$ is an isomorphism of tensor functors, condition (ii) is vacuous. (Cf. \cite[Def. 2.4.8]{TensorCategories:EtingofGelakiNikshychOstrik}.)


\begin{defi}
	A tensor category $(C,\otimes)$ is rigid, if for every object $X$ of $C$ there exists an object $X^\vee$ (called a \emph{dual} of $X$) of $C$  together with morphisms $\operatorname{ev}\colon X\otimes X^\vee\to \mathds{1}$ and $\delta\colon\mathds{1}\to X^\vee\otimes X$ such that
	
		$$
	X  \xrightarrow{r} X\otimes \mathds{1}\xrightarrow{\id\otimes\delta} X\otimes (X^\vee\otimes X)\xrightarrow{\Phi} (X\otimes X^\vee)\otimes X\xrightarrow{\operatorname{ev}\otimes \id} \mathds{1}\otimes X\xrightarrow{l^{-1}} X $$
	and 	
	$$ X^\vee  \xrightarrow{l} \mathds{1}\otimes X^\vee\xrightarrow{\delta\otimes \id}(X^\vee\otimes X)\otimes X^\vee\xrightarrow{\Phi^{-1}} X^\vee\otimes (X\otimes X^\vee)\xrightarrow{\id\otimes\operatorname{ev}}X^\vee\otimes\mathds{1}\xrightarrow{r^{-1}} X^\vee
	$$
	are the identity morphism.
\end{defi}

In \cite[Section 2]{Deligne:categoriestannakien} it is shown that the above definition of rigid is equivalent to the one used in \cite{DeligneMilne:TannakianCategories}.  In a rigid tensor category the dual $X^\vee$ of an object $X$ is uniquely determined up to an isomorphism. A rigid tensor category $(C,\otimes)$ is \emph{abelian} if $C$ is an abelian category. In this case $\otimes$ is automatically bi-additive (\cite[Prop. 1.16]{DeligneMilne:TannakianCategories}).

Let $k$ be a field and $(C,\otimes)$ a rigid abelian tensor category. An isomorphism between $k$ and $\operatorname{End}(\mathds{1})$ induces the structure of a $k$-linear category on $C$ such that $\otimes$ is $k$-bilinear. (See the discussion after Def. 1.15 in \cite{DeligneMilne:TannakianCategories}.)

Let $R$ be a ring. A basic example of a tensor category is the category $\operatorname{Mod}_R$ of all $R$-modules with the usual tensor product $(M_1,M_2)\rightsquigarrow M_1\otimes_R M_2$ of $R$-modules.
The associativity constraint $\Phi$ is given by $$\Phi_{M_1,M_2,M_3}\colon M_1\otimes_R(M_2\otimes_R M_3)\to (M_1\otimes_R M_2)\otimes_R M_3,\ m_1\otimes(m_2\otimes m_3)\mapsto (m_1\otimes m_2)\otimes m_3$$ and the commutativity constraint $\Psi$ is given by $\Psi_{M_1,M_2}\colon M_1\otimes_R M_2\to M_2\otimes_R M_1,\ m_1\otimes m_2\mapsto m_2\otimes m_1$. Any free module $U$ of rank one with an isomorphism $u\colon U\to U\otimes_R U$ is an identity object. Note that any identity object $(U,u)$ can be written in the form $U=Ru_0$ and $u\colon U\to U\otimes_RU,\ u_0\mapsto u_0\otimes u_0$ for some basis element $u_0$ of $U$.

Let $k$ be a field. Recall that $\vec_k$ is the category of all finite dimensional vector spaces over $k$. As explained above, this is a tensor category. In fact, $(\vec_k,\otimes_k)$ is rigid and abelian. For a $k$-algebra $R$, the functor $\vec_k\to \operatorname{Mod}_R,\ V\rightsquigarrow V\otimes_k R$ together with the isomorphism of functors $c$ with components $c_{V,W}\colon (V\otimes_k R)\otimes_R (W\otimes_k R)\simeq (V\otimes_k W)\otimes_k R$ is a tensor functor.

\begin{defi}
	Let $k$ be a field. A \emph{neutral tannakian category over $k$} is a rigid abelian tensor category $(C,\otimes)$ together with an isomorphism $k\simeq \operatorname{End}(\mathds{1})$ such that there exists an exact $k$-linear tensor functor $\omega\colon(C,\otimes)\to (\vec_k,\otimes_k)$. Any such functor is called a \emph{(neutral) fibre functor}.
\end{defi}
We note that a fibre functor is faithful by \cite[Cor. 2.10]{Deligne:categoriestannakien}. Let $(C,\otimes)$ be a neutral tannakian category over the field $k$ and $\omega\colon C\to\vec_k$ a fibre functor. For a $k$-algebra $R$, we denote the composition of $\omega$ with the tensor functor $\vec_k\to \operatorname{Mod}_R,\ V\rightsquigarrow V\otimes_k R$  by $\omega_R$. In particular, $\omega_R(X)=\omega(X)\otimes_k R$ for ever object $X$ of $C$. We define $$\autt(\omega)(R)=\Aut^\otimes(\omega_R)$$
as the group of tensor automorphisms (i.e., invertible morphisms of tensor functors) of the tensor functor $\omega_R\colon (C,\otimes)\to (\operatorname{Mod}_R,\otimes_R)$. This assignment is functorial in $R$ and so $\autt(\omega)$ is a functor from the category of $k$-algebras to the category of groups.


\medskip

Let $G$ be a proalgebraic group over $k$. Recall (from the beginning of this section) that $\Rep(G)$ denotes the category of finite dimensional representations of $G$. With the associativity and commutativity constraint induced from $\vec_k$, $\Rep(G)$ is naturally a neutral tannakian category over $k$ with fibre functor $(V,\f)\rightsquigarrow V$. We are now prepared to state the main tannaka reconstruction theorem.

\begin{theo}[{\cite[Theorem 2.11]{DeligneMilne:TannakianCategories}}] \label{theo: tannaka main}
	Let $k$ be a field, $C$ a neutral tannakian category over $k$ and $\omega\colon C\to \vec_k$ a fibre functor. Then $G=\autt(\omega)$ is a proalgebraic group over $k$ and $\omega$ defines a tensor equivalence between $C$ and $\Rep(G)$.
\end{theo}

\begin{rem} \label{rem: recover G}
	If we choose $C=\Rep(G)$ for a some proalgebraic group $G$ over a field $k$ and $\omega\colon \Rep(G)\to \vec_k,\ (V,\f)\rightsquigarrow V$ in Theorem \ref{theo: tannaka main}, then $\autt(\omega)\simeq G$ by \cite[Prop. 2.8]{DeligneMilne:TannakianCategories}.
\end{rem}

\subsection{Tensor skeletons}

We introduce a tensor analog of skeletal categories and show that every tensor category is tensor equivalent to a tensor skeletal tensor category. Besides the cardinality issue already mentioned in the introduction, there is another reason why it is important to work with skeletons: Two categories are ought to be considered the ``same'' if they are equivalent. However, the eyes of model theory are conditioned to recognize the stronger notion of isomorphic categories. Since two categories are equivalent if and only if they have isomorphic skeletons, the two points of view can be reconciled by considering skeletons. (Cf. the remark at the very end of this section.)

\begin{defi} \label{defi: unique otimes factorization}
	Let $(C,\otimes)$ be a tensor category. An object $V$ of $C$ is \emph{tensor irreducible} if it is not in the image of $\otimes\colon C\times C\to C$, i.e., $V$ is not equal to $V_1\otimes W_1$ for objects $V_1$ and $W_2$ of $C$. The tensor category $(C,\otimes)$ has the \emph{unique tensor factorization property} if
	\begin{enumerate}
		\item $V_1\otimes W_1=V_2\otimes W_2$ implies $V_1=V_2$ and $W_1=W_2$ for objects, $V_1,V_2,W_1,W_2$ of $C$, i.e., $\otimes$ is injective on objects and
		\item every object of $C$ is a finite tensor product of tensor irreducible objects.
	\end{enumerate}
\end{defi}

The following example shows that a tensor category may not have any tensor irreducible objects. It also gives an example of a tensor category that does not satisfy condition (i) of the above definition.
\begin{ex} \label{ex: tensor on skeleton}
	Let $k$ be a field and let $C$ be the category whose objects are the $k$-vector spaces $k^n$ $(n\geq 0)$. The morphisms of $C$ are all $k$-linear maps between any two objects of $C$. For $m,n\geq 0$ choose an isomorphism $\eta_{m,n}\colon k^m\otimes_k k^n\to k^{mn}$.

	 On objects we define $\otimes \colon C\times C\to C$ by $k^m\otimes k^n=k^{mn}$ and on morphisms we define $\otimes$ through the isomorphisms $\eta_{m,n}$, i.e., such that
	$$
	\xymatrix{
			k^{m_1}\otimes_k k^{n_1} \ar^{f_1\otimes_k f_2}[r]  \ar_{\eta_{m_1n_1}}[d] &  k^{m_2}\otimes_k k^{n_2} \ar^{\eta_{m_2n_2}}[d] \\
	k^{m_1n_1} \ar^{f_1\otimes f_2}[r] & k^{m_2n_2}  
}
$$
commutes 
for any $f_1\colon k^{m_1}\to k^{n_1}$ and $f_2\colon k^{m_2}\to k^{n_2}$. Similarly, we can use the isomorphisms $\eta_{m,n}$ to define associativity and commutativity constraints. If fact, $C$ then becomes a neutral tannakian category over $k$. Note that no object of $C$ is tensor irreducible. For example, $k^n=k^n\otimes k$ for every $n\geq 0$.

We note that $C$ is a skeleton of the neutral tannakian category $\vec_k$ and that some choice (namely the choice of the $\eta_{m,n}$) was involved to define a tensor product on the skeleton $C$. Using tensor skeletal tensor categories we will be able to avoid this choice. See Example \ref{ex: tensor skeletal} below.
\end{ex}

\begin{lemma} \label{lemma: unique tensor factorization}
	Let $(C,\otimes)$ be a tensor category with the unique tensor factorization property. Then every object of $C$ is uniquely the tensor product of a finite completely parenthesized sequence of tensor irreducible objects of $C$.
\end{lemma}
\begin{proof}
	We only have to establish the uniqueness. Let $V$ be an object of $C$. If $V$ is tensor irreducible the claim is obvious. So we may assume that $V$ is not tensor irreducible. Assume we have two presentations of $V$ as tensor products of tensor irreducible objects. This yields two presentations $V=V_1\otimes W_1=V_2\otimes W_2$, where either $V_1$ or $W_1$ is tensor irreducible and either $V_2$ or $W_2$ are tensor irreducible. Without loss of generality, let us assume that $V_1$ is tensor irreducible. Then, by condition (i) of Definition \ref{defi: unique otimes factorization}, we find that $V_1=V_2$ and $W_1=W_2$. So $V_2=V_1$ is tensor irreducible and we have two presentations of $W_1=W_2$ as tensor products of tensor irreducible objects. Applying the same reasoning to $W_1=W_2$ and iterating this process, we reach, after finitely many steps, a situation where $W_1=W_2$ is tensor irreducible.
\end{proof}

\begin{defi}
Let $(C,\otimes)$ be a tensor category with the unique tensor factorization property. The \emph{tensor length} of an object $V$ of $C$ is the length of the unique completely parenthesized sequence of tensor irreducible objects whose tensor product equals $V$. 
\end{defi}	
In particular, the tensor irreducible objects of $C$ are those of tensor length one.

\begin{defi}
	A tensor category $(C,\otimes)$ is \emph{tensor skeletal} if
	\begin{enumerate}
		\item any two isomorphic tensor irreducible objects are equal,
		\item every object of $C$ is isomorphic to a tensor irreducible object and
		\item $C$ has the unique tensor factorization property.
	\end{enumerate}
\end{defi}
In a tensor skeletal tensor category every isomorphism class contains a unique tensor irreducible object and any object is uniquely the tensor product of tensor irreducible objects. In particular, the tensor irreducible objects are a skeleton of the category $C$.


\begin{prop} \label{prop: existence of tensor skeleton}
	Let $(C,\otimes)$ be a tensor category. Then there exists a tensor skeletal tensor category $(C',\otimes')$ and a strict tensor equivalence $(C',\otimes')\to (C,\otimes)$. 
\end{prop}
\begin{proof}
	We first choose a skeleton $S$ for the category $C$, i.e., $S$ is a full subcategory of $C$ such that every object of $C$ is isomorphic to a unique object in $S$. In the next step we close $S$ under the tensor product in a generic fashion. In detail, we define the category $C'$ as follows. Let $P$ denote all 
	%
	%
	completely parenthesized finite sequences of objects of $S$. Note that for a finite sequence $V_1,\ldots, V_n$ of objects in $C$ there are $\frac{1}{n} {{2(n-1)}\choose{n-1}}$ possible ways to completely parenthesize the sequence. (This is the $(n-1)$-st Catalan number.) For example, for $n=4$, we have the following possibilities
	$$ (((V_1,V_2),V_3),V_4) \quad ((V_1,V_2),(V_3,V_4)) \quad ((V_1,(V_2,V_3)),V_4),\quad ((V_1,(V_2,V_3)),V_4) \quad (V_1,((V_2,V_3),V_4)).$$
	To every $p\in P$ we associate on object $V(p)$ of $C$ by evaluating the parenthesized sequence via $\otimes$. For example, for $p=(((V_1,V_2),V_3),V_4)$, we have $V(p)=((V_1\otimes V_2)\otimes V_3)\otimes V_4$. The class of objects of $C'$ is defined as all pairs $(p,V(p))$, where $p\in P$. A morphism in $C'$ from $(p,V(p))$ to $(q,V(q))$ is a morphism in $C$ from $V(p)$ to $V(q)$.

	%
	%
	
	We define a tensor product $\otimes'\colon C'\times C'\to C'$ as follows. On objects we set $$(p,V(p))\otimes'(q,V(q))=(pq,V(p)\otimes V(q))=(pq,V(pq)),$$ where $pq$ denotes the concatenation of two parenthesized sequences. For example, for $p=(((V_1,V_2),V_3),V_4)$ and $q=((W_1, W_2),W_3)$ we have $pq=((((V_1,V_2),V_3),V_4),((W_1,W_2),W_3)).$
	
	For morphisms $f\colon (p_1,V(p_1))\to(p_2,V(p_2))$ and $g\colon (q_1,V(q_1))\to (q_2,V(q_2))$ in $C'$ we define $$f\otimes' g\colon (p_1,V(p_1))\otimes'(q_1,V(q_1))=(p_1q_1,V(p_1)\otimes V(q_1))\to (p_2q_2,V(p_2)\otimes V(q_2))=(p_2,V(p_2))\otimes' (q_2,V(q_2))$$ as $f\otimes g$. 
	We define an associativity constraint $\Phi'$ for $\otimes'$ by defining
	$$\Phi'_{(p,V(p)),(q,V(q)),(r,V(r))}\colon (p,V(p))\otimes' ((q,(V(q))\otimes' (r,V(r)))\to ((p,V(p))\otimes' (q,(V(q)))\otimes' (r,V(r))$$
	as $\Phi_{V(p),V(q),V(r)}\colon V(p)\otimes (V(q)\otimes V(r))\to (V(p)\otimes V(q))\otimes V(r).$
	Similarly, we define a commutativity constraint $\Psi'$ for $\otimes'$ by defining $$\Psi'_{(p,V(p)),(q,V(q))}\colon (p,V(p))\otimes'(q,V(q))\to (q,V(q))\otimes' (p,V(p))$$
	as $\Psi_{V(p),V(q)}\colon V(p)\otimes V(q)\to V(q)\otimes V(p)$.
	The commutative diagrams for $\Phi$ and $\Psi$ yield the corresponding commutative diagrams for $\Phi'$ and $\Psi'$. Moreover, if $\mathds{1}$ together with $u\colon \mathds{1}\to \mathds{1}\otimes\mathds{1}$ is an identity object for $C$, then $((\mathds{1}),\mathds{1})$ together with $u'\colon ((\mathds{1}),\mathds{1})\to ((\mathds{1}),\mathds{1})\otimes' ((\mathds{1}),\mathds{1})$, defined as $u$ is an identity object for $C'$. Thus $(C',\otimes')$ is a tensor category.
	
	Let us show that $(C',\otimes')$ is tensor skeletal. Clearly, the tensor irreducible objects of $C'$ are exactly those of the form $((V),V)$, where $V$ belongs to $S$. Since $S$ is a skeleton of $C$, it follows that every object of $C'$ is isomorphic to a tensor irreducible object and that any two isomorphic tensor irreducible objects are equal. By construction, every object of $C'$ is a tensor product of finitely many tensor irreducible objects. For $p_1,p_2,q_1,q_2\in P$ with $p_1q_1=p_2q_2$ we have $p_1=p_2$ and $q_1=q_2$. It follows that condition (i) of Definition \ref{defi: unique otimes factorization} is satisfied. Thus $(C',\otimes')$ is a tensor skeletal tensor category.
	
	The functor $T\colon C'\to C,\ (p,V(p))\rightsquigarrow V(p)$ is clearly fully faithful. Since $S$ is a skeleton it is an equivalence of categories. Moreover, $T$ is a strict tensor functor by construction.
	%
\end{proof}

\begin{ex} \label{ex: tensor skeletal}
	Let $S$ denote the skeleton of $C=\vec_k$ consisting of all vector spaces of the form $k^n$, $(n\geq 0)$. The proof of Proposition \ref{prop: existence of tensor skeleton} yields a tensor skeletal tensor category $C'$ whose objects are in bijection with the completely parenthesized finite sequences of $k^n$'s. The tensor product $\otimes'$ for $C'$ is induced from the tensor product on $\vec_k$. No additional choices an in Example \ref{ex: tensor on skeleton} are required. 
\end{ex}

Let $(C,\otimes)$ be a neutral tannakian category over a field $k$. Then the tensor product of two objects of $C$ that are not zero objects is not a zero object (e.g., because this property holds in $\Rep(G)$, for any proalgebraic group $G$). It follows that the full subcategory of all objects of $C$ that are not zero objects is stable under the tensor product and therefore is naturally a tensor category.

\begin{defi}
	A neutral tannakian category $(C,\otimes)$ over a field $k$ is \emph{pointed skeletal} if it has exactly one zero object and the full subcategory of all objects that are not zero objects is tensor skeletal.
\end{defi}

We note that in a pointed skeletal neutral tannakian category $(C,\otimes)$ the zero object is not tensor irreducible. Moreover, an object of $C$, different from the zero object, is tensor irreducible in $C$ if and only if it is tensor irreducible in the tensor skeletal tensor category of all objects that are not the zero object. Every object of $C$, different from the zero object, is uniquely the tensor product of a finite completely parenthesized sequence of tensor irreducible objects. As before, we call the length of this sequence the tensor length of the object. The tensor irreducible objects together with the zero object form a skeleton of $C$.

We need to introduce the above notion for a rather technical reason: If one attempts to work with tensor skeletal neutral tannakian categories one runs into trouble with axiom (\ref{item: Existence of tensor factorization}) below, because for the zero object, this axiom does not seem to be expressible as a first order statement. We will need a version of Proposition \ref{prop: existence of tensor skeleton} for neutral tannakian categories.

\begin{cor} \label{cor: exists pointed skeleton}
	Let $(C,\otimes)$ be a neutral tannakian category over $k$. Then there exists a pointed skeletal neutral tannakian category $(C',\otimes')$ over $k$ and a $k$-linear tensor equivalence $(C',\otimes)\to (C,\otimes)$.
\end{cor}
\begin{proof}
	Let $(D,\otimes)$ denote the full subcategory of $(C,\otimes)$ consisting of all objects that are not zero objects. Applying Proposition \ref{prop: existence of tensor skeleton} to the tensor category $(D,\otimes)$ yields a tensor category $(D',\otimes')$ together with a strict tensor equivalence $T\colon (D',\otimes')\to(D,\otimes)$. Note that $D'$ does not have a zero object because $T$ is an equivalence of categories and $D$ does not have a zero object.

	We extend the category $D'$ to a category $C'$ by adding a zero object $\boldsymbol{0}$. So the objects of $C'$ are the disjoint union of the objects of $D'$ with $\bold{0}$. The morphisms between two objects in $C'$ that both belong to $D'$ are the same as the morphisms in $D'$.
	For an object $V'$ of $C'$ there is a unique morphism $0\colon \bold{0}\to V'$ in $C'$. Similarly, there is a unique morphism $0\colon V'\to\bold{0}$. Composition of morphisms in $C'$ is defined such that composition with a zero morphism always yields a zero morphism. For example, the composition $V'\to\bold{0}\to W'$ is the unique $f\in\Hom(V',W')$ such that $T(f)\colon T(V')\to T(W')$ is the zero morphism (in $C$).

We extend the functor $T\colon D'\to D$ to a functor $T\colon C'\to C$ by choosing $T(\boldsymbol{0})$ to be a zero object	of $C$ and by defining $T$ of a zero morphism to be a zero morphism. Then the functor $T\colon C'\to C$ defines an equivalence of categories. Since $C$ is abelian, it follows that also $C'$ is abelian.

	Next we extend $\otimes'\colon D'\times D'\to D'$ to a functor $\otimes'\colon C'\times C'\to C'$ in the only meaningful way. Namely, $\bold{0}\otimes V'=\bold{0}$ and $V'\otimes\bold{0}=\bold{0}$ for every object $V'$ of $C'$. Similarly, $f\otimes 0=0$ and $0\otimes f=0$ for any morphism $f$ in $C'$. (Here $0$ denotes an appropriate zero morphism.) The associativity and commutativity constraints on $D'$ extend trivially to associativity and commutativity constraints on $C'$. So $(C',\otimes')$ is a tensor category. Moreover, $T\colon (C',\otimes')\to (C,\otimes)$ is a tensor equivalence. Since $(C,\otimes)$ is rigid, it follows that also $(C',\otimes')$ is rigid. For an identity object $\mathds{1}'$ of $C'$ we have $\operatorname{End}(\mathds{1}')\simeq \operatorname{End}(T(\mathds{1}'))\simeq k$. For the induced $k$-linear structure on $C'$ the functor $T$ is $k$-linear. Composing $T$ with a fibre functor $\omega\colon C\to \vec_k$ yields a fibre functor for $C'$. Thus $(C',\otimes')$ is a neutral tannakian category over $k$. By construction $(C',\otimes')$ is pointed skeletal.
\end{proof}

The following lemma is needed in the next subsection to define the category $\TANN$.

\begin{lemma}\label{lemma: tensor skeletal implies small}
	Let $k$ be a field and $C$ a neutral tannakian category over $k$. If $C$ is pointed skeletal, then $C$ is small, i.e., the class of objects of $C$ is a set. 
\end{lemma}
\begin{proof}
	Let $\omega\colon C\to \vec_k$ be a fibre functor and set $G=\autt(\omega)$. According to Theorem \ref{theo: tannaka main} we have an equivalence of categories $C\to \Rep(G)$. Every representation of $G$ is isomorphic to a representation of $G$ on $k^n$ for some $n\geq 0$. Thus the class of objects of a skeleton of $\Rep(G)$ is a set. Since $C$ is pointed skeletal, the class of all tensor irreducible objects of $C$ together with the zero object is a skeleton of $C$. Since equivalent categories have isomorphic skeletons, it follows that the class of tensor irreducible objects of $C$ is a set. Since every object of $C$, different from the zero object, is a finite tensor product of tensor irreducible objects it follows that the class of objects of $C$ is a set.
\end{proof}

\subsection{$\TANN$: The category of neutral tannakian categories}

Our goal is to study proalgebraic groups from a model theoretic perspective by axiomatizing their categories of representations. The models of our theory $\PROALG$ will correspond to pointed skeletal neutral tannakian categories with a fibre functor. The models of $\PROALG$ together with the homomorphisms, i.e., the structure preserving maps, form a category that is equivalent to a certain category of neutral tannakian categories that we now describe in detail.

We define the category $\TANN$ as follows. The objects of $\TANN$ are triples $(k,C,\omega)$, where $k$ is a field, $C$ is a pointed skeletal neutral tannakian category over $k$ and $\omega\colon C\to\vec_k$ is a fibre functor. We note that by Lemma \ref{lemma: tensor skeletal implies small} pointed skeletal neutral tannakian categories are small, so there is no set theoretic obstruction to forming this category, like the obstruction one encounters when attempting to form the category of all categories.

A morphism in $\TANN$ from $(k,C,\omega)$ to $(k',C',\omega')$ is a pair $(\lambda,T,\alpha)$, where $\lambda\colon k\to k'$ is a morphism of fields, $T\colon C\to C'$ is a $k$-linear strict tensor functor that preserves tensor irreducible objects and $\alpha\colon \omega_{k'}\to\omega' T$ is an isomorphism of tensor functors. Here $\omega_{k'}\colon C\to \vec_{k'}$ denotes the tensor functor obtained by composing $\omega$ with the tensor functor $\vec_k\to \vec_{k'},\ V\rightsquigarrow V\otimes_k{k'}$ induced by $\lambda\colon k\to k'$.


The composition of two morphisms $(\lambda,T,\alpha)\colon(k,C,\omega)\to (k',C',\omega')$ and $(\lambda',T',\alpha')\colon (k',C',\omega')\to  (k'',C'',\omega'')$ in $\TANN$ is the pair $(\lambda'\lambda,T'T,\gamma)\colon (k,C,\omega)\to (k'',C'',\omega'')$, where $\gamma\colon\omega_{k''}\to \omega''T'T$ is given by
$$\gamma_V\colon\omega(V)\otimes_k k''=(\omega(V)\otimes_k k')\otimes_{k'} k''\xrightarrow{\alpha_V\otimes k''}\omega'(T(V))\otimes_{k'}k''\xrightarrow{\alpha'_{T(V)}}\omega''(T'(T(V))$$ for every object $V$ of $C$.

We also  define a category $\PRO$. The objects are pairs $(k,G)$, where $k$ is a field and $G$ a proalgebraic group over $k$. A morphism $(\lambda,\f)\colon (k,G)\to (k', G')$ in \linebreak[4] $\PRO$ is a pair $(\lambda,\f)$, where $\lambda\colon k\to k'$ is a morphism of fields and $\f\colon G'\to G_{k'}$ is a morphism of proalgebraic groups over $k'$. Here $G_{k'}$ is the base change of $G$ from $k$ to $k'$ via $\lambda$. The composition $(\lambda'', \f'')\colon (k,G)\to (k',G'')$ of two morphism $(\lambda,\f)\colon (k,G)\to (k', G')$ and $(\lambda',\f')\colon (k',G')\to (k'', G'')$ is defined by $\lambda''=\lambda'\lambda$ and $\f''\colon G''\xrightarrow{\f'}G'_{k''}\xrightarrow{\f_{k''}}(G_{k'})_{k''}=G_{k''}$.

The following proposition is essential for establishing the close relationship between models of $\PROALG$ and proalgebraic groups.

\begin{prop}  \label{prop: from TANN to PROALGEBRAIC GROUPS} The functor $(k,C,\omega)\rightsquigarrow (k, \autt(\omega))$ from the category $\TANN$ to the category $\PRO$ is full, essentially surjective and induces a bijection on the isomorphism classes.
\end{prop}
\begin{proof}
	Let $(k,C,\omega)$ be an object of $\TANN$. From Theorem \ref{theo: tannaka main} we know that $G=\underline{\Aut}^\otimes(\omega)$ is a proalgebraic group over $k$ and so we obtain an object $(k,G)$ of $\PRO$. A morphism $(\lambda,T,\alpha)\colon (k,C,\omega)\to (k',C',\omega')$ in $\TANN$ defines a morphism $(\lambda,\f)\colon (k,G)\to (k',G')$ in $\PRO$ as follows: Let $R'$ be a $k'$-algebra and $g'\in G'(R')=\operatorname{Aut}^\otimes(\omega'_{R'})$. So for every object $V'$ of $C'$, we have an $R'$-linear automorphism $g'_{V'}\colon\omega'(V')\otimes_{k'}R'\to \omega'(V')\otimes_{k'}R'$. We define an element $\f(g')\in G_{k'}(R')=G(R')=\operatorname{Aut}^\otimes(\omega_{R'})$ by
	$$\f(g')_V\colon \omega(V)\otimes_k{R'}\simeq (\omega(V)\otimes_k k')\otimes_{k'}{R'}\xrightarrow{\alpha_V\otimes R'}\omega'(T(V))\otimes_{k'}R'\xrightarrow{g'_{T(V)}}\omega'(T(V))\otimes_{k'}R'\simeq\omega(V)\otimes_k{R'}$$
	for every object $V$ of $C$. Then $\f_{R'}\colon G'(R')\to G_{k'}(R'),\ g'\mapsto \f(g')$ is a morphism of groups that is functorial in $R'$ and therefore defines a morphism $\f\colon G'\to G_{k'}$ of proalgebraic groups over $k'$. We thus have a functor from $\TANN$ to $\PRO$.
	
	Let us show that this functor is full. We assume that a morphism $(\lambda,\f)\colon (k, G)\to (k',G')$ is given. We will define a morphism $(\lambda, T,\alpha)\colon (k,C,\omega)\to (k',C',\omega')$ that induces $(\lambda,\f)$. Let us first explain the idea for the construction of $T$: We have tensor functors
	\begin{equation} \label{eqn: functor for inverse} C\xrightarrow{\omega}\Rep(G)\to \Rep(G_{k'})\to \Rep(G'),\end{equation}
	where $\Rep(G)\to \Rep(G_{k'})$ is given by $V\rightsquigarrow V\otimes_k k'$ and $\Rep(G_{k'})\to \Rep(G')$ is the restriction via $\f\colon G'\to G_{k'}$. Composing the functor (\ref{eqn: functor for inverse}) with a quasi-inverse of the tensor equivalence $C'\xrightarrow{\omega'} \Rep(G')$ yields a functor $T\colon C\to C'$. However, it is a priori not clear that a quasi-inverse can be chosen in such a way that $T$ is a strict tensor functor that preserves tensor irreducible objects. Moreover, the construction of $T$ is intertwined with the construction of $\alpha$.

	To define $T$ and $\alpha$, consider first a tensor irreducible object $V$ of $C$. The representation of $G=\underline{\Aut}^\otimes(\omega)$ on $\omega(V)$, induces a representation of $G_{k'}$ on $\omega(V)\otimes_k k'$ and by restriction via $\f\colon G'\to G_{k'}$ we obtain a representation of $G'$ on $\omega(V)\otimes_k k'$. By Theorem \ref{theo: tannaka main} the category $C'$ is equivalent (via $\omega'$) to the category of representations of $G'$. Thus there exists an object $T(V)$ of $C'$ and an isomorphism $\alpha_V\colon\omega(V)\otimes_k k'\to \omega'(T(V))$ of representations of $G'$. In fact, since $C'$ is tensor skeletal, we may choose $T(V)$ to be tensor irreducible.
	

	 This defines $T$ on tensor irreducible objects. To define $T$ on an object $V$ of $C$, different from the zero object and of tensor length $n\geq 2$, we may assume that $T$ has already been defined on objects of tensor length less than $n$. We know from Lemma \ref{lemma: unique tensor factorization} that $V$ is uniquely of the form $V=V_1\otimes V_2$, where $V_1$ and $V_2$ have tensor length less than $n$. We can thus define $T(V)$ as $T(V)=T(V_1)\otimes'T(V_2)$. Finally, we define $T$ of the zero object of $C$ to be the (unique) zero object of $C'$. 
%
This completes the definition of $T$ on objects. Note that we have $T(V\otimes W)=T(V)\otimes' T(W)$ for all objects $V, W$ of $C$.
	
	We extend the definition of $\alpha$ in a similar manner: We have already defined $\alpha_V$ for tensor irreducible objects $V$. Let $V$ be an object of $C$ of tensor length $n$ and assume $\alpha_V$ has been defined on objects of tensor length less than $n$. As $V$ is of the form $V=V_1\otimes V_2$ with $V_1$ and $V_2$ of tensor length less then $n$, we can define $\alpha_V=\alpha_{V_1\otimes V_2}$ as the unique map making
	\begin{equation} \label{eqn: diagram for definition of alpha} 
	\xymatrix{
		\omega(V_1\otimes V_2)\otimes_k k' \ar^{\alpha_{V_1\otimes V_2}}[r] \ar_\simeq[d] & \omega'(T(V_1\otimes V_2)) \ar^\simeq[d] \\
		(\omega(V_1)\otimes_k k')\otimes_{k'}(\omega(V_2)\otimes_k k') \ar^-{\alpha_{V_1}\otimes\alpha_{V_2}}[r] & \omega'(T(V_1))\otimes_{k'}(\omega'(T(V_2))}
	\end{equation}
	commutative. For the zero object $V$, $\alpha_V$ is defined as the zero map. Then, by construction, the above diagram commutes for any pair of objects $V_1$ and $V_2$ of $C$.
	
	We next define $T$ on morphisms. Let $f\colon V\to W$ be a morphism in $C$. We then have a morphism $\omega(f)\colon \omega(V)\to\omega(W)$ of representations of $G$, that induces a morphism $\omega(f)\otimes k'\colon \omega(V)\otimes_k k'\to\omega(W)\otimes_k k'$ of representations of $G_{k'}$. This is also a morphism of representations of $G'$. In fact, we have morphisms of representations of $G'$
	$$
	\xymatrix{
		\omega(V)\otimes_k k' \ar^{\omega(f)\otimes k'}[r] \ar_{\alpha_V}[d] & \omega(W)\otimes_k k' \ar^{\alpha_W}[d] \\
		\omega'(T(V)) & \omega'(T(W))
	}
	$$
	where the vertical maps are isomorphisms. Since $\omega'$ induces an equivalence of categories, there exists a unique morphism $T(f)\colon T(V)\to T(W)$ in $C'$ such that
	\begin{equation} \label{eqn: diagram T on morphisms}
		\xymatrix{
		\omega(V)\otimes_k k' \ar^{\omega(f)\otimes k'}[r] \ar_{\alpha_V}[d] & \omega(W)\otimes_k k' \ar^{\alpha_W}[d] \\
		\omega'(T(V)) \ar^{\omega'(T(f))}[r] & \omega'(T(W))
	}
	\end{equation}
	commutes. This completes the definition of $(\lambda, T, \alpha)$. Let us check that $T$ is indeed a strict tensor functor. To see that $T$ is compatible with the associativity constraint, let $U,V,W$ be objects of $C$ and $\Phi_{U,V,W}\colon U\otimes(V\otimes W)\to (U\otimes V)\otimes W$ the corresponding associativity isomorphism. We have the following commutative diagram:
	
	$$
	\xymatrix{
		\omega(U\otimes (V\otimes W))\otimes_k{k'}  \ar^{\omega(\Phi_{U,V,W})\otimes k'}[r] \ar_\simeq[d] & \omega((U\otimes V)\otimes W)\otimes_k k' \ar^\simeq[d] \\
		(\omega(U)\otimes_k k')\otimes_{k'}\big((\omega(V)\otimes_k k')\otimes_{k'}(\omega(W)\otimes_k k')\big) \ar[r] \ar_{\alpha_U\otimes(\alpha_V\otimes \alpha_W)}[d] & 	\big((\omega(U)\otimes_k k')\otimes_{k'}(\omega(V)\otimes_k k')\big)\otimes_{k'}(\omega(W)\otimes_k k') \ar^{(\alpha_U\otimes\alpha_V)\otimes \alpha_W}[d] \\
		\omega'(T(U))\otimes_{k'}\big(\omega'(T(V))\otimes_{k'}\omega'(T(W))\big) \ar[r] \ar_\simeq[d] & \big(\omega'(T(U))\otimes_{k'}\omega'(T(V))\big)\otimes_{k'}\omega'(T(W)) \ar^\simeq[d] \\
		\omega'(T(U)\otimes(T(V)\otimes T(W))) \ar_=[d] \ar^{\omega'(\Phi'_{T(U),T(V),T(W)})}[r] & \omega'((T(U)\otimes T(V))\otimes T(W)) \ar^=[d] \\
		\omega'(T(U\otimes(V\otimes W))) \ar@{..>}^{\omega'(T(\Phi_{U,V,W}))}[r] & \omega'(T((U\otimes V)\otimes W))
	}	
	$$
	
	Thanks to the commutativity of (\ref{eqn: diagram for definition of alpha}), we know that the map from the upper left to the lower left corner is $\alpha_{U\otimes(V\otimes W)}$. Similarly, the map from the upper right to the lower right corner is $\alpha_{(U\otimes V)\otimes W}$. Since, by definition, $T(\Phi_{U,V,W})$ is the unique morphism such that $\omega'(T(\Phi_{U,V,W}))$ makes the outer rectangle of the above diagram commute, we conclude that $T(\Phi_{U,V,W})=\Phi'_{T(U),T(V),T(W)}$ as desired. 
In a similar fashion one shows that $T(\Psi_{U,V})=\Psi'_{T(U),T(V)}$.	
Since $T$ preserves identity objects we conclude that $T$ is a strict tensor functor. Moreover, the commutativity of (\ref{eqn: diagram T on morphisms}) shows that $T$ is $k$-linear and by construction $T$ preserves tensor irreducible objects.

The commutativity of (\ref{eqn: diagram T on morphisms}) also shows that $\alpha\colon \omega_{k'}\to \omega'T$ is an isomorphism of functors and the commutativity of (\ref{eqn: diagram for definition of alpha}) shows that $\alpha$ is an isomorphism of tensor functors.
Thus $(\lambda,T,\alpha)$ is indeed a morphism in $\TANN$.

 As the $\alpha_V$'s are morphisms of representations of $G'$ and $G'$ is acting on $\omega(V)\otimes_k k'$ through the restriction via $\f\colon G'\to G_{k'}$, it is then clear that the morphism $(\lambda,T,\alpha)$ induces the morphism $(\lambda,\f)$ we started with. Thus the functor  $(k,C,\omega)\rightsquigarrow (k, \autt(\omega))$ is full.

We next show that it is essentially surjective. Let $G$ be a proalgebraic group over a field $k$. Applying Corollary \ref{cor: exists pointed skeleton} to the neutral tannakian category $\Rep(G)$ yields a pointed skeletal neutral tannakian category $(C,\otimes)$ and a $k$-linear tensor equivalence $F\colon C\to \Rep(G)$. We define a fibre functor $\omega\colon C\to\vec_k$ by composing $F$ with the forgetful functor $\omega_G\colon\Rep(G)\to\vec_k$. Then $(k,C,\omega)$ is an object of $\TANN$. 
Moreover, since $F$ is a tensor equivalence, the natural morphism  of functors $\autt(\omega_G)\to \autt(\omega)$ is an isomorphism. Since $\autt(\omega_G)$ is isomorphic to $G$ (Remark \ref{rem: recover G}) we see that $(k,C,\omega)\rightsquigarrow (k, \autt(\omega))$ is essentially surjective.

\medskip

Finally, we establish the bijection on isomorphism classes. Since we already proved the essential surjectivity, it suffices to show the following: For objects $(k,C,\omega)$ and $(k',C',\omega')$ of $\TANN$, if $(k,\autt(\omega))$ and $(k',\autt(\omega'))$ are isomorphic, then $(k,C,\omega)$ and $(k',C',\omega')$ are isomorphic. We abbreviate $G=\autt(\omega)$ and $G'=\autt(\omega')$. Let $(\lambda,\f)\colon (k, G)\to (k',G')$ be an isomorphism in $\PRO$. In the above proof that $(k,C,\omega)\rightsquigarrow (k, \autt(\omega))$ is full, we have already seen how to construct a morphism $(\lambda, T,\alpha)\colon (k,C,\omega)\to(k',C',\omega')$ from $(\lambda,\f)$. We claim that $(\lambda, T,\alpha)$ is an isomorphism.

We first show that $T$ is surjective on objects. Let $V'$ be a tensor irreducible object of $C'$. Then $\omega'(V')$ is a representation of $G'$. By assumption $\lambda\colon k\to k'$ is an isomorphism of fields and $\f\colon G'\to G_{k'}$ is an isomorphism of proalgebraic groups. In the sequel we will use $\lambda^{-1}\colon k'\to k$ to base change from $k'$ to $k$. For example, $\omega'(V')\otimes_{k'}k$ is a representation of $G'_k$. But $G'_{k}$ is isomorphic to $G$ via $G_{k'}\xrightarrow{\f_k} (G_{k'})_k\simeq G$ and so we can consider  $\omega'(V')\otimes_{k'}k$ to be a representation of $G$. Since $\omega\colon C\to\Rep(G)$ is an equivalence of categories, there exists an object $V$ of $C$ such that $\omega(V)$ is isomorphic to $\omega'(V')\otimes_{k'}k$ as a representation of $G$. Moreover, since $C$ is tensor skeletal, we can choose $V$ to be tensor irreducible. It follows that $\omega(V)\otimes_k k'$ is isomorphic to $\omega'(V')$ as a representation of $G'$. As $T(V)$ is, by definition, the unique tensor irreducible object of $C'$ such that $\omega'(T(V))$ is isomorphic to $\omega(V)\otimes_k k'$ as a representation of $G'$, it follows that $T(V)=V'$.
So $T$ is surjective on tensor irreducible objects. An arbitrary object $V'$ of $C'$, different from the zero object, is a finite tensor product of tensor irreducible objects. We can choose an inverse image under $T$ for all these tensor irreducible objects, form their tensor product in $C$ and then apply the strict tensor functor $T$ to see that $V'$ is in the image of $T$.

We next show that $T$ is injective on objects. First let $V_1$ and $V_2$ be tensor irreducible objects of $C$ such that $T(V_1)=T(V_2)$. Then $\omega'(T(V_1))=\omega'(T(V_2))$ as representation of $G'$ and $\omega'(T(V_1))\otimes_{k'}k=\omega'(T(V_2))\otimes_{k'} k$ as representation of $G$. But $\omega'(T(V_1))\otimes_{k'}k\simeq \omega(V_1)$ as representation of $G$ and similarly for $V_2$. So $\omega(V_1)$ and $\omega(V_2)$ are isomorphic representations of $G$. But then $V_1$ and $V_2$ must be isomorphic objects of $C$. Since $V_1$ and $V_2$ are tensor irreducible, it follows that $V_1=V_2$. Thus $T$ is injective on tensor irreducible objects. From the uniqueness in Lemma \ref{lemma: unique tensor factorization} it then follows that $T$ is injective on objects.

Using diagram (\ref{eqn: diagram T on morphisms}), we see that $T$ is fully faithful. Since $T$ is bijective on objects, $T$ is an isomorphism of categories, i.e., there exists a functor $T^{-1}\colon C'\to C$ such that $TT^{-1}=\id_{C'}$ and $T^{-1}T=\id_C$. Since $T$ is a strict tensor functor, also $T^{-1}$ is a strict tensor functor. Similarly, as $T$ preserves tensor irreducible objects, also $T^{-1}$ preserves tensor irreducible objects. Finally, $T^{-1}\colon C'\to C$ is $k'$-linear, where the $k'$-linear structure on $C$ is defined via $\lambda^{-1}\colon k'\to k$. 

For every object $V'$ of $C'$ we have a $k'$-linear isomorphism $\alpha_{T^{-1}(V')}\colon \omega(T^{-1}(V'))\otimes_k k'\to \omega'(V')$. The base change of this map via $\lambda^{-1}\colon k'\to k$ is a $k$-linear isomorphism $\alpha_{T^{-1}(V')}\otimes k\colon\omega(T^{-1}(V'))\to \omega'(V')\otimes_{k'}k$.  We define $(\alpha^{-1})_{V'}=(\alpha_{T^{-1}(V')}\otimes k)^{-1}$. Then $\alpha^{-1}\colon \omega'_k\to \omega T^{-1}$ is an isomorphism of functors. Since $\alpha$ is an isomorphism of tensor functors, also $\alpha^{-1}$ is an isomorphism of tensor functors. Finally, $(\lambda^{-1},T^{-1},\alpha^{-1})$ is an inverse to $(\lambda,T,\alpha)$ in $\TANN$.
%
\end{proof}

The functor of Proposition \ref{prop: from TANN to PROALGEBRAIC GROUPS} is not faithful. This is reflected in the proof of the fullness of the functor by the fact that the $\alpha_V$'s, for $V$ tensor irreducible, can be chosen arbitrarily.

We note that for Proposition \ref{prop: from TANN to PROALGEBRAIC GROUPS} to be valid it is important to consider pointed skeletal neutral tannakian categories. For example, the neutral tannakian categories in Example \ref{ex: tensor on skeleton} and \ref{ex: tensor skeletal} both correspond to the trivial proalgebraic group. However, these two categories are not isomorphic.

\section{$\PROALG$: Neutral tannakian categories as first order structures}

\label{section: Neutral tannakian categories as first order structures}

In this section we define a many-sorted first order theory PROALG such that the isomorphism classes of models of PROALG are in bijection with the isomorphism classes of proalgebraic groups. The idea is to axiomatize pointed skeletal neutral tannakian categories with a fibre functor.

\subsection{The language}

We define a many-sorted first order theory PROALG as follows:

\vspace{5mm}

\noindent\emph{Sorts:}

\medskip

We have three different types of sorts: The \emph{field} sort, the \emph{objects} sorts and the \emph{morphisms} sorts. The objects sorts and the morphisms sorts split further into the \emph{base} objects/morphisms sorts and the \emph{total} objects/morphisms sorts. We will use the following notation:

With $k$ we denote the universe of the field sort. For every pair 
$p=(m,n)$ of integers $m,n\geq 1$ we have two objects sorts: The base objects sort with universe $B_p$ and the total objects sort with universe $X_p$. For every pair $p$, $q$, where, as above $p=(m,n)$ and $q=(m',n')$ with $m,n,m',n'\geq 1$, we have two morphisms sorts: The base morphisms sort with universe $B_{p,q}$ and the total morphisms sort with universe $X_{p,q}$.

The idea is that $B_p$, where $p=(m,n)$, represents $k$-vector spaces of tensor length $m$ and dimension $n$, considered as objects of a category, i.e., every element of $B_p$ corresponds to such a vector space. On the other hand, $X_p$ contains the actual vector spaces.

Similarly, for morphisms: $B_{p,q}$ represents morphisms from vector spaces in $B_p$ to vector spaces in $B_q$; every element of $B_{p,q}$ corresponds to a morphism, the actual linear maps are encoded in $X_{p,q}$. 

%
%
%

\vspace{5mm}

\noindent\emph{Constant symbols:}

\begin{itemize}
	\item We have two constant symbols $0$ and $1$ for the field-sort.
\end{itemize}

\vspace{5mm}

\noindent\emph{Relation symbols:}

\begin{itemize}
	\item For every $p$ we have a unary relation symbol $0_p$ on $X_p$. 
	\item For every $p$ we have a ternary relation symbol $A_p$ on $X_p$. (The ``A'' is for addition.)
	\item We have a constant symbol $1$ in $B_{(1,1)}$.
	\item For every $p=(m,n)$ we have an $n$-ary relation symbol $LI_p$ on $X_p$. (``LI'' is for linear independence.)
\end{itemize}

\vspace{5mm}

\noindent\emph{Function symbols:}

\begin{itemize}
	\item We have two binary function symbols $+$ and $\cdot$ for the field-sort.
	\item For every $p$ we have a function symbol $\pi_p$ with interpretation $\pi_p\colon X_p\to B_p$. 
	\item For every pair $p,q$ we have a function symbol $\pi_{p,q}$ with interpretation $\pi_{p,q}\colon X_{p,q}\to B_{p,q}$.
	\item For every $p$ we have a function symbol $SM_p$ with interpretation $SM_p\colon k\times X_p\to X_p$. (``SM'' is for scalar multiplication.)
	\item For all $p,q$ we have function symbols $S^B_{p,q}$ and $T^B_{p,q}$ with interpretations $S^B_{p,q}\colon B_{p,q}\to B_p$ and $T^B_{p,q}\colon B_{p,q}\to B_q$. (``S'' is for source and ``T'' is for target of a morphism.)
	\item  For all $p,q$ we have function symbols $S^X_{p,q}$ and $T^X_{p,q}$ with interpretations $S^X_{p,q}\colon X_{p,q}\to X_p$ and $T^X_{p,q}\colon X_{p,q}\to X_q$.
	\item For $p=(m,n)$ and $q=(m',n')$ with $m,n,m',n'\geq 1$ we set $pq=(m+m',nn')$. We have function symbols $\otimes_{p,q}$ with interpretations $\otimes_{p,q}\colon X_p\times X_q\to X_{pq}$.
\end{itemize}

We denote this many-sorted language with $\mathcal{L}$.


\subsection{The axioms}

Rather than stating the axioms explicitly in the above language, we state their mathematical content. It is however clear that all the axioms below can be expressed as a collection of $\mathcal{L}$-sentences.


\renewcommand{\labelenumi}{{\rm (\arabic{enumi})}}

\begin{enumerate}
	\item $(k,+,\cdot, 0, 1)$ is a field. 
	\item For every $p$, the map $\pi_p\colon X_p\to B_p$ is surjective. To simplify the notation we set $X_p(b)=\pi_p^{-1}(b)$ for $b\in B_p$.
	\item Existence of zero: For every $V=X_p(b)$, (where $b\in B_p$), the set $V\cap 0_p$ has a unique element $0_V$.
	\item Vector space addition: For $v_1,v_2,v_3\in X_p$, if $A_p(v_1,v_2,v_3)$ holds, then $\pi_p(v_1)=\pi_p(v_2)=\pi_p(v_3)$.
	Moreover, for $b=\pi_p(v_1)=\pi_p(v_2)=\pi_p(v_3)\in B_p$, and $V=X_p(b)$, the set $\{(v_1,v_2,v_3)\in V^3|\ A_p(v_1,v_2,v_3)\}$ is the graph of a map  
	$+_V\colon V\times V\to V$, that defines on $V$ the structure of an abelian group with identity element $0_V$. 
	\item Scalar multiplication: If $\lambda\in k$ and $v\in X_p(b)$, for some $b\in B_p$, then also $SM_p(\lambda,a)\in X_p(b)$. Moreover, for every $V=X_p(b)$, the restriction of $SM_p$ to $\cdot_V\colon k\times V\to V$ defines a scalar multiplication on $V$ such that $V$ is a vector space over $k$ with addition $+_V$. 
	\item Dimension: Every $X_p(b)$ ($b\in B_p$) is an $n$-dimensional $k$-vector space, where $p=(m,n)$. 
	\item The maps $\pi_{p,q}\colon X_{p,q}\to B_{p,q}$ are surjective. For $f\in B_{p,q}$ we set $X_{p,q}(f)=\pi_{p,q}^{-1}(f)$.
	\item The diagram
	$$
	\xymatrix{
	X_p \ar_{\pi_p}[d] & X_{p,q} \ar^{\pi_{p,q}}[d] \ar^{T^X_{p,q}}[r] \ar_{S^X_{p,q}}[l] & X_q \ar^{\pi_q}[d] \\
	B_p & B_{p,q} \ar_{S^B_{p,q}}[l] \ar^{T^B_{p,q}}[r] & B_q
	}
	$$
	commutes.
	\item \label{item: Morphisms} Morphisms: For every $f\in B_{p,q}$ the map $S^X_{p,q}\colon X_{p,q}(f)\to X_p(S^B_{p,q}(f))$ is bijective. The image of $(S^X_{p,q},T^X_{p,q})\colon X_{p,q}(f)\to X_p(S^B_{p,q}(f))\times X_q(T^B_{p,q}(f))$, is the graph of a $k$-linear map $\widetilde{f}\colon X_p(S^B_{p,q}(f))\to X_q(T^B_{p,q}(f))$. 
	
	Moreover, if $f,g\in B_{p,q}$ with $S^B_{p,q}(f)=S^B_{p,q}(g)$ and $T^B_{p,q}(f)=T^B_{p,q}(g)$ such that $\widetilde{f}=\widetilde{g}$, then $f=g$.
	
	\item \label{item: Existence of the identity} Existence of the identity: For all $b\in B_p$ there exists an $f\in B_{p,p}$ such that $S^X_{p,p}(a)=T^X_{p,p}(a)$ for all $a\in X_{p,p}(f)$.
	
	\item \label{item: Composition of morphisms} Composition of morphisms: For $f\in B_{p,q}$ and $g\in B_{q,r}$ with $T^B_{p,q}(f)=S^B_{q,r}(g)$ there exists $h\in B_{p,r}$ such that $\widetilde{h}=\widetilde{g}\circ\widetilde{f}$.

	
	\item \label{item: Linearity} Linearity: For $f,g\in B_{p,q}$ with $S^B_{p,q}(f)=S^B_{p,q}(g)$ and $T^B_{p,q}(f)=T^B_{p,q}(g)$, there exists $h\in B_{p,q}$ such that $\widetilde{f}+\widetilde{g}=\widetilde{h}$.
	
	Moreover, for every $f\in B_{p,q}$ and $\lambda\in k$, there exists $g$ in $B_{p,q}$ such that $\lambda\widetilde{f}=\widetilde{g}$.
	(In particular, for $\lambda=0$, we see that the zero morphism is of the form $\widetilde{g}$.)
	
	\item \label{item: tensor compatible with projections} Tensor is compatible with projections: For $a\in X_p$ and $b\in X_q$ we write $a\otimes b$ or $a\otimes_{p,q} b$ for $\otimes_{p,q}(a,b)$.

	If $a_1,a_2\in X_p$ with $\pi_p(a_1)=\pi_p(a_2)$ and $b_1,b_2\in X_q$ with $\pi_q(b_1)=\pi_q(b_2)$, then $\pi_{pq}(a_1\otimes b_1)=\pi_{pq}( a_2\otimes b_2)$.
	
	\item Bilinearity of tensor product: If $a_1,a_2\in X_p$ with $\pi_p(a_1)=\pi_p(a_2)$ and $b\in X_q$, then $(a_1+a_2)\otimes b=a_1\otimes b+a_2\otimes b$. Moreover, for $\lambda\in k$ we have $\lambda a_1\otimes b=\lambda(a_1\otimes b)$. Similarly for left- and right-hand side interchanged.
	
	\item \label{item: Tensor product} Tensor product: 
	For $b\in B_p$ and $c\in B_q$, let $b\otimes c=\pi_{pq}(v\otimes w)\in B_{pq}$, where $v\in X_p(b)$ and $w\in X_q(c)$. (Note that, by axiom \ref{item: tensor compatible with projections}, $b\otimes c$ does not depend on the choice of $v$ and $w$.)
	
	The map  $X_p(b)\otimes_k X_q(c)\to X_{pq}(b\otimes c)$ induced by the bilinear map $\otimes_{p,q}\colon X_p(b)\times X_q(c)\to X_{pq}(b\otimes c)$ is an isomorphism.
We set $X_p(b)\otimes X_q(c)=X_{pq}(b\otimes c)$
	
	%
	
	\item \label{item: Functoriality of tensor product} Functoriality of tensor product: For $b_1\in B_{p_1}$, $b_2\in B_{p_2}$, $c_1\in B_{q_1}$, $c_2\in B_{q_2}$, $f\in B_{p_1,q_1}$ with $S^B_{p_1,q_1}(f)=b_1$ and $T^B_{p_1,q_1}(f)=c_1$, and $g\in B_{p_2,q_2}$ with $S^B_{p_2,q_2}(g)=b_2$ and $T^B_{p_2,q_2}(g)=c_2$, there exists $h\in B_{p_1p_1,q_1q_2}$ such that $$\widetilde{h}=\widetilde{f}\otimes\widetilde{g}\colon X_{p_1p_2}(b_1\otimes b_2)=X_{p_1}(b_1)\otimes_k X_{p_2}(b_2)\to X_{q_1}(c_1)\otimes_k X_{q_2}(c_2)=X_{q_1q_2}(c_1\otimes c_2).$$
%
%
%
	\item \label{item: Associativity of the tensor product} Associativity of the tensor product: For $b\in B_p$, $c\in B_q$ and $d\in B_r$, there exists $f\in B_{pqr,pqr}$ such that $\widetilde{f}\colon X_p(b)\otimes(X_q(c)\otimes X_r(d))\to (X_p(b)\otimes X_q(c))\otimes X_r(d)$ equals the map defined by $u\otimes (v \otimes w)\mapsto (u\otimes v)\otimes w$.
	
	\item \label{item: Commutativity of the tensor product} Commutativity of the tensor product: For $b\in B_p$ and $c\in B_q$, there exists $f\in B_{pq,pq}$ such that $\widetilde{f}\colon X_p(b)\otimes X_q(c)\to X_q(c)\otimes X_p(b)$ equals the map defined by $v\otimes w\mapsto w\otimes v$.
	
	\item  \label{item: Uniqueness of tensor factorization}
	Uniqueness of tensor factorization: If $p_1q_1=p_2q_2$, $b_1\in B_{p_1}$, $c_1\in B_{q_1}$, $b_2\in B_{p_2}$ and $c_2\in B_{q_2}$ are such that $b_1\otimes c_1=b_2\otimes c_2$, then $p_1=p_2$, $p_2=q_2$, $b_1=b_2$ and $c_1=c_2$.	
%
	
	\item \label{item: Existence of tensor factorization}
	
	Existence of tensor factorization: For every $b\in B_{(m,n)}$, there exist elements $b_1,\ldots,b_m$ with $b_i\in B_{(1,n_i)}$ for some $n_i$ with $n_1\ldots n_m=n$ and a complete parenthesization of the sequence $b_1,\ldots,b_m$ such that the corresponding tensor product of the sequence is equal to $b$.
	
	\item \label{item: tensor skeletal} Tensor skeletal: For $b\in B_{(m,n)}$, there exists a  $c\in B_{(1,n)}$ and $f\in B_{((m,n),(1,n))}$ such that $\widetilde{f}\colon X_{(m,n)}(b)\to X_{(1,n)}(c)$ is bijective.
	
	Moreover, if $b,c\in B_{(1,n)}$ are such that there exists an $f$ in $B_{(1,n),(1,n)}$ with $\widetilde{f}\colon X_{(1,n)}(b)\to X_{(1,n)}(c)$ bijective, then $b=c$.

	\item \label{item: Existence of the identity object} Existence of the identity object: Recall that $1$ is a constant in $B_{(1,1)}$. For every non-zero element $u_0$ of $\mathds{1}=X_{(1,1)}(1)$ and $b\in B_{(m,n)}$, there exists $f$ in $B_{((m,n),(m+1,n)}$ such that $\widetilde{f}\colon X_{(m,n)}(b)\to \mathds{1}\otimes X_{(m,n)}(b)$ is the map $v\mapsto u_0\otimes v$.
	\item \label{item: Existence of duals} Existence of duals: For every $b\in B_{(m,n)}$ there exists $b^\vee$ in $B_{(1,n)}$, $f\in B_{(m+1,n^2),(1,1)}$ and $g\in B_{(1,1),(m+1,n^2)}$ with $\widetilde{f}\colon V\otimes V^\vee\to \mathds{1}$ and $\widetilde{g}\colon \mathds{1}\to V^\vee\otimes V$, where $V=X_{(m,n)}(b)$ and $V^\vee=X_{(1,n)}(b^\vee)$, such that 
	the maps 
	$$
	V  \to V\otimes \mathds{1}\xrightarrow{V\otimes\widetilde{g}} V\otimes (V^\vee\otimes V)\to (V\otimes V^\vee)\otimes V\xrightarrow{\widetilde{g}\otimes V} \mathds{1}\otimes V\to V $$	
	$$ V^\vee  \to \mathds{1}\otimes V^\vee\xrightarrow{\widetilde{g}\otimes V^\vee}(V^\vee\otimes V)\otimes V^\vee\to V^\vee\otimes (V\otimes V^\vee)\xrightarrow{V^\vee\otimes\widetilde{f}}V^\vee\otimes\mathds{1}\to V^\vee
	$$
	are the identity maps.
	
	\item \label{item: Existence of direct sums} Existence of direct sums (biproducts): For $b\in B_{(m,n)}$ and $c\in B_{(m',n')}$, there exists $d\in B_{(1,n+n')}$, $P_b\in B_{((1,n+n'),(m,n))}$, $P_c\in B_{(1,n+n'),(m',n')}$, $I_b\in B_{(m,n),(1,n+n')}$ and $I_c\in B_{(m',n'),(1,n+n')}$ such that $\widetilde{I_b}\circ\widetilde{P_b}+\widetilde{I_c}\circ\widetilde{P_c}=\id_{X_{(1,n+n')}(d)}$, $\widetilde{P_b}\circ\widetilde{I_b}=\id_{X_{(m,n)}(b)}$, $\widetilde{P_c}\circ\widetilde{I_c}=\id_{X_{(m',n')}}$, $\widetilde{P_c}\circ\widetilde{I_b}=0$ and $\widetilde{P_b}\circ\widetilde{I_c}=0$.

	\item \label{item: Existence of kernels} Existence of kernels: For every $f$ in $B_{p,q}$ and $f'$ in $B_{r,q}$ with $\widetilde{f}\colon V\to W$ injective and $\widetilde{f'}\colon U\to W$ such that $\widetilde{f'}(U)\subseteq \widetilde{f}(V)$, there exists $f''\in B_{r,p}$ with $\widetilde{f''}\colon U\to V$ such that $\widetilde{f}\circ\widetilde{f''}=\widetilde{f'}$.
	
	Moreover, for every $f\in B_{p,q}$ with $\widetilde{f}\colon V\to W$ and $\dim(\ker(\widetilde{f}))=\ell\geq 1$, there exists $f'\in B_{(1,\ell),p}$ with $\widetilde{f'}\colon U\to V$ such that $\widetilde{f'}$ is injective and $\widetilde{f}\circ{\widetilde{f'}}=0$.

	\item \label{item: Existence of cokernels} Existence of cokernels: For $f\in B_{p,q}$ and $f'\in B_{p,r}$ with $\widetilde{f}\colon V\to W$ surjective and $\widetilde{f'}\colon V\to U$ such that $\ker(\widetilde{f})\subseteq\ker(\widetilde{f'})$, there exists $f''\in B_{q,r}$ with $\widetilde{f''}\colon W\to U$ such that $\widetilde{f''}\circ\widetilde{f}=\widetilde{f'}$. 
	
	Moreover, for every $f\in B_{p,q}$ with $\widetilde{f}\colon V\to W$ and $\dim(\im(\widetilde{f}))=\ell<n$, where $q=(m,n)$, there exists $\widetilde{f'}\in B_{(q,(1,n-\ell))}$, with $\widetilde{f'}\colon W\to U$, such that $\widetilde{f'}$ is surjective and $\widetilde{f'}\circ\widetilde{f}=0$.

	\item Linear independence: For $v_1,\ldots,v_n\in X_{p}$, we have $LI_{p}(v_1,\ldots,v_n)$, where $p=(m,n)$, if and only if $\pi_p(v_i)=\pi_p(v_j)$ for $1\leq i,j\leq n$ and $v_1,\ldots,v_n$ are linearly independent (in $X_p(b)$, where $b=\pi_p(v_i)$).

\end{enumerate}

\renewcommand{\labelenumi}{{\rm (\roman{enumi})}}

\begin{rem}
	The relations $LI_p$ are definable from the other symbols of the language. So the relation symbols $LI_p$ could in principle be omitted from the language. It is however convenient to work with the $LI_p$'s, because they imply that a homomorphism of models of $\operatorname{PROALG}$ has certain desirable properties. See Theorem \ref{theo: equivalence of categories} and its proof for details.
\end{rem}

\subsection{Equivalence of $\PROALG$ and $\TANN$}

Let $\M$ and $\M'$ be models of $\operatorname{PROALG}$. Recall that a homomorphism $h\colon \M\to \M'$ is a sequence of maps, one for each sort $s$, that maps the $\M$-universe of the $s$-sort to the $\M'$-universe of the $s$-sort, such that all constants, relations, and functions are preserved.

\begin{theo} \label{theo: equivalence of categories}
	The category of models of $\operatorname{PROALG}$ with the homomorphisms as morphisms is equivalent to the category $\TANN$.
\end{theo}
\begin{proof}
	Let $\mathcal{M}=(k,B_p,X_p,B_{p,q},X_{p,q}) $ be an object of $\operatorname{PROALG}$. We will associate an object $(k(\M), C(\mathcal{M}),\omega(\mathcal{M}))$ of $\TANN$ to $\mathcal{M}$. We set $k(\M)=k$ (including the field structure) and we define a category  $D=D(\mathcal{M})$ as follows: The set of objects of $D$ is the disjoint union of all $B_p$'s. For $b\in B_p$ and $c\in B_q$, the set of morphisms from $b$ to $c$ is defined as $$\Hom(b,c)=\{f\in B_{p,q}|\ S^B_{p,q}(f)=b,\ T^B_{p,q}(f)=c\}.$$ To define the composition
	$g\circ f$ for $f\in\Hom(b,c)$ and $g\in\Hom(c,d)$ we us axioms \ref{item: Morphisms} and \ref{item: Composition of morphisms}: We define $g\circ f$ as the unique element of $B_{p,r}$ such that $\widetilde{f\circ g}=\widetilde{g}\circ\widetilde{f}$. As the composition of $k$-linear maps is associative, it follows that our composition is also associative. For $b\in B_p$ we define the identity $\id_b\in\Hom(b,b)$ to be the unique element of $\Hom(b,b)$ with $\widetilde{\id_b}=\id_{X_p(b)}$ (axiom \ref{item: Existence of the identity}). It is then clear that $D$ is a category.
	
	We define a tensor product $\otimes\colon  D\times D\to D$ as follows: On objects, say $b\in B_p$ and $c\in B_q$, we define $b\otimes c\in B_{pq}$ as in axiom \ref{item: Tensor product}. For morphisms $f\colon b_1\to c_1$ and $g\colon b_2\to c_2$ 
	we define $f\otimes g\colon b_1\otimes b_2\to c_1\otimes c_2$ as the unique element of $\Hom(b_1\otimes b_2,c_1\otimes c_2)$ with $\widetilde{f\otimes g}=\widetilde{f}\otimes\widetilde{g}$ (axiom \ref{item: Functoriality of tensor product}). It is then clear that $\otimes$ is a functor.
	
	We define an associativity constraint $\Phi$ with components $\Phi_{b,c,d}\colon b\otimes (c\otimes d)\to (b\otimes c)\otimes d$ such that $\widetilde{\Phi_{b,c,d}}$ corresponds to the usual associativity constraint in $\vec_k$ (axiom \ref{item: Associativity of the tensor product}). Similarly, we define a commutativity constraint $\Psi$ using axiom \ref{item: Commutativity of the tensor product}. The required diagrams for $\Phi$ and $\Psi$ commute because the corresponding diagrams commute in $\vec_k$.
	 
	Let $u_0$ be any non-zero element of $\mathds{1}$ (see axiom \ref{item: Existence of the identity object}) and let $u$ be the unique element of $\Hom(1,1\otimes 1)$ such that $\widetilde{u}\colon \mathds{1}\to \mathds{1}\otimes\mathds{1},\ u_0\mapsto u_0\otimes u_0$.
	 We claim that $(1,u)$ is an identity object for $(D,\otimes)$. Note that by axiom \ref{item: Existence of kernels} every $f\in B_{p,q}$ such that $\widetilde{f}$ is bijective is an isomorphism. So it follows from axiom \ref{item: Existence of the identity object} that $b\rightsquigarrow \mathds{1}\otimes b$ is an equivalence of categories.
	%
	Thus $D$ is a tensor category. Axioms \ref{item: Uniqueness of tensor factorization}, \ref{item: Existence of tensor factorization} and \ref{item: tensor skeletal} imply that $D$ is tensor skeletal. The tensor irreducible objects are those belonging to some $B_{(1,n)}$.
	
	Note that the category $D$ does not have a zero object, because all the $X_p(b)'$s are vector spaces of dimension greater or equal to $n\geq 1$. We now add a zero object $\bold{0}$ to $D$ to form a category $C=C(\M)$. This is done in a similar fashion as in the proof of Corollary \ref{cor: exists pointed skeleton}. So the objects of $C$ are the disjoint union of the objects of $D$ with $\bold{0}$. The morphisms between two objects in $C$ that both belong to $D$ are the same as the morphisms in $C$.
	For an object $b$ of $D$ there is a unique morphism $0\colon \bold{0}\to b$ in $C$. Similarly, there is a unique morphism $0\colon b\to\bold{0}$. Composition of morphisms in $C$ is defined in the obvious way. For example, the composition $b\to\bold{0}\to c$ is the unique $f\in\Hom(b,c)$ such that $\widetilde{f}$ is the zero map (cf. axiom \ref{item: Linearity}).
	
	 For consistence reasons we extend some of our notation to include zero: We define $X(\bold{0})$ to be the zero vector space (over $k$) and we set $\widetilde{0}$ to be the zero map.
		
	As in the proof of Corollary \ref{cor: exists pointed skeleton} we extend $\otimes\colon D\times D\to D$ to a functor $\otimes\colon C\times C\to C$ in the only meaningful way. The associativity and commutativity constraints on $D$ extend trivially to associativity and commutativity constraints on $C$. So $(C,\otimes)$ is a tensor category. It follows from axiom \ref{item: Existence of duals} that $(C,\otimes)$ is rigid.
	
	Let us next show that $C$ is an abelian category. Clearly $C$ has a zero object, namely $\bold{0}$. By axiom \ref{item: Existence of direct sums} the category $C$ has biproducts. It follows from axioms	
	\ref{item: Existence of kernels} and \ref{item:  Existence of cokernels} that $C$ has kernels and cokernels.
	Let $f\colon b\to c$ be a monomorphism in $C$. We claim that $\widetilde{f}\colon X_p(b)\to X_q(c)$ is injective. For a kernel $g\colon a\to b$ of $f$ we have $fg=0=f0$ and therefore $g=0$. Since the image of $\widetilde{g}$ is the kernel of $\widetilde{f}$ we see that $\widetilde{f}$ is injective. It then follows from axiom \ref{item: Existence of cokernels} that $f$ is the kernel of its cokernel. So $f$ is normal. Similarly, we see that also every epimorphism in $C$ is normal. Thus $C$ is an abelian category.
	
	For $f\in \Hom(b,c)$ and $\lambda\in k$ we define $\lambda f$ as the unique element of $\Hom(b,c)$ such that $\widetilde{\lambda f}=\lambda\widetilde{f}$ (axiom \ref{item: Linearity}). Thus $C$ becomes a $k$-linear category.

	%

	We now define the fibre functor $\omega=\omega(\mathcal{M})$ from $C$ to $\vec_k$: For $b\in B_p$ we set $\omega(b)=X_p(b)$. For a morphism 
$f\colon b\to c$ in $C$ we set $\omega(f)=\widetilde{f}$. By construction $f$ is an exact $k$-linear functor.	The isomorphism $\gamma\colon \omega(-)\otimes_k\omega(-)\to \omega(-\otimes -)$ of functors, with components
	$$\gamma_{b,c}\colon \omega(b)\otimes_k\omega(c)=X_p(b)\otimes_k X_p(c)\to X_{pq}(b\otimes c)=\omega(b\otimes c),\ v\otimes w\to v\otimes_{p,q} w$$ turns $\omega$ into a tensor functor (cf. axiom \ref{item: Tensor product}). Thus $C$ is a neutral tannakian category over $k$. Since $(D,\otimes)$ is tensor skeletal, we see that, as desired, $C$ is a pointed skeletal neutral tannakian category over $k$. Thus $(k(\M),C(\mathcal{M}),\omega(\mathcal{M}))$ is an object of $\TANN$.
	
	
	\medskip
	
	Let $h\colon\mathcal{M}\to\mathcal{M}'$ be a homomorphism of models of $\operatorname{PROALG}$. Where $\mathcal{M}=(k,B_p,X_p,B_{p,q},X_{p,q})$, $\mathcal{M}'=(k',B'_p,X'_p,B'_{p,q},X'_{p,q})$ and $h=(h_{\text{field}}, h_p^B, h_p^X, h_{p,q}^B, h_{p,q}^X)$

	We claim that $h$ induces a morphism $(\lambda(h),T(h),\alpha(h))\colon (k, C, \omega)\to (k', C', \omega')$ in $\TANN$. We set $\lambda(h)=h_{\text{field}}$ and consider $k'$ as a field extension of $k$ via $\lambda(h)$.
	
	 We define the functor $T=T(h)\colon C\to C'$ through the action of $h$ on the base sorts: We set $T(b)=h^B_p(b)$ for $b\in B_p$ and $T(f)=h^B_{p,q}(f)$ for $f\in B_{p,q}$. We also set $T(\bold{0})=\bold{0}'$ and $T(0)=0$. The commutativity of
	 $$
	 \xymatrix{
	 B_p \ar_{h_p}[d] & B_{p,q} \ar_{S^B_{p,q}}[l] \ar^{T^B_{p,q}}[r] \ar^{h_{p,q}}[d] & B_q \ar^{h_q}[d] \\
	 B'_p & B'_{p,q} \ar_{S^B_{p,q}}[l] \ar^{T^B_{p,q}}[r] & B'_q 
	}
	 $$
	 shows that $T$ maps $\Hom(b,c)$ into $\Hom(T(b),T(c))$. The commutativity of
	 $$
	 \xymatrix{
	X_p \ar^{h_p^X}[r] \ar_{\pi_p}[d] & X_p' \ar^{\pi_p}[d]  \\
	B_p \ar^{h_p^B}[r] & B_p' 	
	 }
	$$	 
	 shows that $h_p^X$ maps $X_p(b)$ into $X'_p(T(b))$. Similarly, the commutativity of 	
	  $$
	 \xymatrix{
	 	X_{p,q} \ar^-{h_{p,q}^X}[r] \ar_{\pi_{p,q}}[d] & X_{p,q}' \ar^{\pi_{p,q}}[d]  \\
	 	B_{p,q} \ar^-{h_{p,q}^B}[r] & B_{p,q}' 	
	 }
	 $$	 
	 shows that $h^X_{p,q}$ maps $X_{p,q}(f)$ into $X_{p,q}(T(f))$.	 
	 For a morphism $f\colon b\to c$ in $C$ the left, right, upper and other squares in 
	 $$
	 \xymatrix{
	 	X_{p,q}(f) \ar@/_3pc/_{T^X_{p,q}}[dd] \ar^-{h^X_{p,q}}[r] \ar_{S^X_{p,q}}[d] & X'_{p,q}(T(f))  \ar^{S^X_{p,q}}[d]  \ar@/^3pc/^{T^X_{p,q}}[dd] \\ 
	 	X_p(b) \ar^-{h^X_p}[r] \ar_{\widetilde{f}}[d] & X'_p(T(b)) \ar^{\widetilde{T(f)}}[d] \\
	 	X_q(c) \ar^-{h^X_q}[r] & X'_q(T(c))  		
	 }
	 $$
	 commute. Thus also the lower square 
	  \begin{equation} \label{eqn: diag commutes with C(f)}
	 \xymatrix{
	 	X_p(b) \ar^-{h^X_p}[r] \ar_{\widetilde{f}}[d] & X'_p(T(b)) \ar^-{\widetilde{T(f)}}[d] \\
	 	X_q(c) \ar^-{h^X_q}[r] & X_q'(T(c))		
	 }
	 \end{equation}
	 commutes.
	 Since $h$ preserves $A_p$ and $SM_p$ it follows that the maps 
	 	$h_p^X\colon X_p(b)\to  X'_p(T(b))$ are $k$-linear. Furthermore, since $h$ preserves $LI_p$ and $X_p(b)$ and $X'_p(T(b))$ have the same dimension, we see that the induced map $\alpha_{b}\colon X_p(b)\otimes_k k'\to X'_p(T(b))$ is an isomorphism of $k'$-vector spaces. Diagram (\ref{eqn: diag commutes with C(f)}) extends to 
	 	\begin{equation} \label{eqn: diag for alpha}
	 	\xymatrix{
	 		X_p(b)\otimes_k k' \ar^-{\alpha_b}[r] \ar_{\widetilde{f}\otimes k'}[d] & X'_p(T(b)) \ar^-{\widetilde{T(f)}}[d] \\
	 		X_q(c)\otimes_k k' \ar^-{\alpha_{c}}[r] & X_q'(T(c))		
	 	}
	 	\end{equation}
	 	a diagram of $k'$-linear maps.

	 This implies that $T(\id_b)=\id_{T(b)}$. Moreover,
	 for morphisms $f\colon b\to c$ and $g\colon c\to d$ in $C$, the commutativity of the diagram
	 $$
	 \xymatrix{
	 X_p(b) \ar@/^2pc/^{\widetilde{gf}}[rr] \ar^{\widetilde{f}}[r] \ar_{h^X_p}[d] & X_q(c) \ar^{\widetilde{g}}[r] \ar^{h^X_q}[d] & X_r(d)  \ar^{h^X_r}[d] \\	
	 X'_p(b) \ar@/_2pc/_{\widetilde{T(g)T(f)}}[rr]  \ar^{\widetilde{T(f)}}[r] & X'_q(c) \ar^{\widetilde{T(g)}}[r] & X'_r(d) 
	 }
	 $$
	 shows that $T(gf)=T(g)T(f)$. Thus $T$ is a functor. We claim that $T$ is a strict tensor functor.
	 
	 Since the diagram	  
	 \begin{equation} \label{eqn: tensor and C(f)}
	 \xymatrix{
	 	X_p(b)\times X_q(c) \ar_{h^X_p\times h^X_q}[d]  \ar^-{\otimes_{p,q}}[r] & X_{pq}(b\otimes c) \ar^{h^X_{pq}}[d] \\
	 	X'_p(T(b))\times X'_q(T(c)) \ar^-{\otimes_{p,q}}[r] & X'_{pq}(T(b)\otimes T(c)) 	
	 }
	 \end{equation}
	 commutes, we see that $T(b\otimes c)=T(b)\otimes T(c)$. Note that the above diagram can also be expressed as $h^X_p\otimes h^X_q=h^X_{pq}$.
The diagram
$$
\xymatrix{
	X(b\otimes(c\otimes d)) \ar^{\widetilde{\Phi_{b,c,d}}}[r] \ar@{=}[d] & X((b\otimes c)\otimes d) \ar@{=}[d] \\
X(b)\otimes (X(c)\otimes X(d)) \ar[r] \ar_{h^X\otimes (h^X\otimes h^X)}[d] & 	(X(b)\otimes X(c))\otimes X(d) \ar^{(h^X\otimes h^X)\otimes h^X}[d] \\
X'(T(b))\otimes (X'(T(c))\otimes X'(T(d))) \ar[r] \ar@{=}[d] & 	(X'(T(b))\otimes X'(T(c)))\otimes X'(T(d)) \ar@{=}[d] \\
X'(T(b\otimes (c\otimes d))) \ar^{\widetilde{\Phi'_{T(b),T(c),T(d)}}}[r] & X'(T((b\otimes c)\otimes d)))
}
$$	 
commutes, where for simplicity we have omitted the $p,q,r$ indices. By (\ref{eqn: tensor and C(f)}) the map from the upper left to the lower right corner is $h^X$. Similarly, the map from the upper right corner to the lower right corner is $h^X$. It thus follows that $T(\Phi_{b,c,d})=\Phi'_{T(b),T(c),T(d)}$. In a similar fashion one shows that $T(\Psi_{b,c})=\Psi'_{T(b),T(c)}$.
	 Since $h^B_{(1,1)}(1)=1'$ we have $T(1)=1'$ and so $T$ is a strict tensor functor. From diagram (\ref{eqn: diagram for definition of alpha}) it follows that $T$ is $k$-linear. Clearly $T$ preserves tensor irreducible objects. In summary, as desired, $T$ is a $k$-linear strict tensor functor that preserves tensor irreducible objects.
	 	
	To obtain a morphism in $\TANN$, we also need to specify an isomorphism $\alpha=\alpha(h)\colon \omega_{k'}\to\omega'T$ of tensor functors. But the collection of all $\alpha_b\colon X_p(b)\otimes_k k'\to X'_p(T(b))$ defined above exactly yields such an isomorphism: The commutativity of (\ref{eqn: diag for alpha}) shows that $\alpha$ is a morphism of functors, whilst the commutativity of (\ref{eqn: tensor and C(f)}) implies that $\alpha$ is an isomorphism of tensor functors. So we indeed have a functor from $\operatorname{PROALG}$ to $\TANN$.

	We next show that the functor $\M\rightsquigarrow (k(\M),C(\M),\omega(\M))$ is faithful. Let $h,g\colon \M\to \M'$ be homomorphisms such that $(\lambda(h),T(h),\alpha(h))=(\lambda(g),T(g),\alpha(g))$. Then $h_{\text{field}}=\lambda(f)=\lambda(g)=g_{\text{field}}$. Moreover, $h^B_p=g^B_p$ and $h^B_{p,q}=g^B_{p,q}$ for all $p,q$ since $T(h)=T(g)$. 
	
	On $X_p(b)$, $h^X_p$ agrees with $\alpha(h)_b$ and $g^X_p$ agrees with $\alpha(g)_b$. Thus $h^X_p=g^X_p$. To show that also $h^X_{p,q}=g^X_{p,q}$, consider $f\in B_{(p,q)}$.
	The vertical maps in the diagram
	$$\xymatrix{
		X_{p,q}(f) \ar@{..>}[r] \ar_{S^X_{p,q}}[d] & X'_{p,q}(T(f)) \ar^{S^X_{p,q}}[d] \\
		X_p(b) \ar^-{h^X_p=g^X_p}[r] & X'_p(T(b))	 	
	}$$
	are bijective. Thus there exists a unique map $	X_{p,q}(f)\to  X'_{p,q}(T(f))$ that makes this diagram commutative. As the restrictions to $X_{p,q}(f)$ of both, $h^X_{p,q}$ and $g^X_{p,q}$ indeed make this diagram commutative, we see that $h^X_{p,q}=g^X_{p,q}$. 
	
	To show that the functor $\M\rightsquigarrow (k(\M),C(\M),\omega(\M))$ is full, consider a morphism $$(\lambda,T,\alpha)\colon (k(\M),C(\M),\omega(\M))\to (k(\M'),C(\M'),\omega(\M'))$$ in $\TANN$. We have to construct a homomorphism $h\colon\M\to\M'$ that induces $(\lambda,T,\alpha)$. We set $h_{\text{field}}=\lambda\colon k\to k'$.
	
	Note that for $b\in B_p=B_{(m,n)}$, the $k$-vector space $\omega(b)=X_{(m,n)}(b)$ has dimension $n$. Since $\alpha_b\colon X_p(b)\otimes_k k'\to  \omega'(T(b))$ is an isomorphism of $k'$-vector spaces we see that $\omega'(T(b))$ also has dimension $n$. Moreover, as $T$ preserves tensor irreducible objects, we see that $T(b)$ has tensor length $m$. So $T(b)\in B'_p$. Thus $T$ induces maps $h^B_p\colon B_p\to B'_p$. Since $B_{p,q}$ is the set of all morphism in $C$ with source in $B_p$ and target in $B_q$, it then also follows that $T$ induces maps $h_{p,q}^B\colon B_{p,q}\to B'_{p,q},\ f\mapsto T(f)$. 

	We define $h^X_p\colon X_p\to X'_p$ by $h^X_p(v)=\alpha_b(v\otimes 1)$, where $b=\pi_p(v)$. To define $h^X_{p,q}$, consider $f\in B_{p,q}$ with $f\colon b\to c$. We define $h^X_{p,q}\colon X_{p,q}\to X'_{p,q}$ to be the unique map whose restriction to any $X_{p,q}(f)$  makes
	\begin{equation} \label{eqn: diag h}
	\xymatrix{
		X_{p,q}(f) \ar@{..>}[r] \ar_{S^X_{p,q}}[d] & X_{p,q}(T(f)) \ar^{S^X_{p,q}}[d] \\
		X_p(b) \ar^{h^X_p}[r] & X'_p(T(b))	 	
	}
\end{equation}
	commutative. We need to check that $h=(h_{\text{field}},h^B_p,h^X_p,h^B_{p,q},h^X_{p,q})$ is a homomorphism. Clearly $f_{\text{field}}$ preserves $+$, $\cdot$, $0$ and $1$. Since the
	$\alpha_b$ are $k'$-linear isomorphisms, $0_p$, $A_p$, $SM_p$ and $LI_p$ are preserved by $h$.

	We note that $C=C(\M)$ has several identity objects. 
	However, since $C$ is pointed skeletal there exists a unique tensor irreducible object $1$ of $C$ such that $(1,u)$ is an identity object, for some isomorphism $u\colon 1\to 1\otimes 1$; similarly for $C'=C(\M')$. Since $T$ preserves identity objects and tensor irreducible objects, we see that $T(1)=1'$, i.e., $h$ preserves $1$.
	
	Since $T$ is a functor, $h$ preserves $S^B_{p,q}$ and $T^B_{p,q}$.
As $\alpha_b\colon X_p(b)\otimes_k k'\to X'_p(T(b))$ we see that $h$ preserves $\pi_p$. Using diagram (\ref{eqn: diag h}) we see that $h$ also preserves $\pi_{p,q}$. Diagram (\ref{eqn: diag h}) shows that $h$ preserves $S^X_{p,q}$.
For a morphism $f\colon b\to c$ in $C$, the
	diagram 
	$$
	\xymatrix{
		X_{p,q}(f) \ar@/^4pc/^{h^X_{p,q}}[rrr]  \ar_-{T^X_{p,q}\otimes 1}[rd] \ar^-{S^X_{p,q}\otimes 1}[r] & X_p(b)\otimes_k k' \ar_-{\widetilde{f}\otimes k'}[d]   \ar^{\alpha_b}[r] & X'_p(T(b)) \ar^-{\widetilde{T(f)}}[d] & \ar_-{S^X_{p,q}}[l] X_{p,q}(T(f)) \ar^-{T^X_{p,q}}[ld] \\	
		& X_q(c)\otimes_k k' \ar^-{\alpha_c}[r] & X'_q(T(c))
	}
	$$
	shows that $f$ also preserves $T^X_{p,q}$.
	As $T$ is strict and $\alpha$ an isomorphism of tensor functors, we have for equivalence classes $b\in B_p$ and $c\in B_q$ a commutative diagram: 
	$$
	\xymatrix{
		(X_p(b)\otimes X_q(c))\otimes_k k' \ar^-{\alpha_{b\otimes c}}[r] \ar_{\simeq}[d] & X'_{pq}(T(b\otimes c)) \ar@{=}[d] \\
		(X_p(b)\otimes_k k')\otimes_{k'}(X_q(c)\otimes_k k') \ar^-{\alpha_b\otimes\alpha_c}[r] & X'_p(T(b))\otimes X'_q(T(c))	
	}
	$$
	
	Thus for $v\in X_p(b)$ and $w\in X_q(c)$ we obtain $h^X_{pq}(v\otimes w)=\alpha_{b\otimes c}(v\otimes w\otimes 1)=\alpha_b(v\otimes 1)\otimes\alpha_c(w\otimes1)=h^X_p(v)\otimes h^X_q(w)$. Thus $h$ also preserves $\otimes_{p,q}$. Therefore $f\colon \M\to \M'$ is a homomorphism of models of PROALG. It is then clear that $h$ induces the morphism $(\lambda,T,\alpha)$ in $\TANN$.
	
	\medskip
	
	Finally, we show that the functor $\M\rightsquigarrow (k(\M),C(\M),\omega(\M))$ is essentially surjective.
	Let $(k,C,\omega)$ be an object of $\TANN$. We will construct a model $\M$ of $\operatorname{PROALG}$ such that $(k(\M),C(\M),\omega(\M))$ is isomorphic to $(k,C,\omega)$ in $\TANN$. We define the field sort of $\M$ to be $k$ (including the field structure). For $p=(m,n)$, with $m,n\geq 1$, let $B_p$ denote the set of all objects $b$ of $C$ of tensor length $m$ and such that $\omega(b)$ has dimension $n$.
	Let $B_{p,q}$ denote the set of all morphisms in $C$ from objects in $B_p$ to objects in $B_q$. Let $S^B_{p,q}\colon B_{p,q}\to B_p$ denote the map that assigns the source to a morphism and similarly for the target. Let $X_p$ denote the disjoint union of the $k$-vector spaces $\omega(b)$, $b\in B_p$ and let $\pi_p\colon X_p\to B_p$ be the map such that $\pi_p(v)=b$ for $v\in\omega(b)$. We define $X_{p,q}$ as the disjoint union of the graphs of the $k$-linear maps $\omega(f)\colon \omega(b)\to\omega(c)$, where $f\colon b\to c$ is a morphism in $C$ with $b\in B_p$ and $c\in B_q$.	The maps $\pi_{p,q}\colon X_{p,q}\to B_{p,q}$ are defined by $\pi_{p,q}(a)=f$, if $a$ belongs to the graph of $\omega(f)$. The maps $S^X_{p,q}\colon X_{p,q}\to X_p$ and $T^X_{p,q}\colon X_{p,q}\to X_q$ are defined by	$S^X_{p,q}((v,w))=v$ and $T^X_{p,q}((v,w))=w$, where $(v,w)=(v,\omega(f)(v))$ is an element of the graph $\{(v,\omega(f)(v))|\ v\in X_p(b)\}$ of $\omega(f)$, for a morphism$f\colon b\to c$ in $B_{p,q}$.
	We define $0_p$ to be the subset of $X_p$ consisting of all zero vectors of all vector spaces in $X_p$. 	
	We define $A_p$ through
	$A_p(v_1,v_2,v_3)$ if and only if all three elements $v_1,v_2,v_3$ of $X_p$ belong to the same vector space $\omega(b)$ and $v_3=v_1+v_2$, where the $+$ here is vector space addition in $\omega(b)$. 
	The map $SM_p\colon k\times X_p\to X_p$ is scalar multiplication.
	 For elements $v_1,\ldots,v_n$ of $X_p$, where $p=(m,n)$, we define $LI_p$ through $LI_p(v_1,\ldots,v_n)$ if and only if all $\pi_p(a_i)=\pi_p(a_j)$ for all $i,j$ and $v_1,\ldots,v_n$ are $k$-linearly independent (in $X_p(b)$, where $b=\pi_p(a_i)$). 
	 
	 There exists a unique tensor irreducible object $\mathds{1}_C$ in $C$ such that $(\mathds{1}_C,u)$ is an identity object of $C$ for some isomorphism $u\colon \mathds{1}_C\to \mathds{1}_C\otimes\mathds{1}_C$. We set $1=\mathds{1}_C\in B_{(1,1)}$.
	
Finally, we define $\otimes_{p,q}\colon X_p\times X_q\to X_{pq}$ by sending $(v,w)\in\omega(b)\times\omega(c)$ to the image of $v\otimes w$ under $\omega(b)\otimes_k\omega(c)\to \omega(b\otimes c)$, where the latter map is part of the functorial isomorphism defining the tensor functor $\omega$. (As $C$ is pointed skeletal, $\otimes_{p,q}$ is well defined.)
	 It is now straight forward to check that our structure $\M$ satisfies all 27 axioms of PROALG. Moreover, $(k(\M),C(\M),\omega(\M))=(k,C,\omega)$. So the isomorphism $(\lambda,T,\alpha)$ can be chosen to be the identity.
\end{proof}

For a model $\M$ of PROALG we define $(k,C,\omega)=(k(\M),C(\M),\omega(\M))$ as in the proof of Theorem \ref{theo: equivalence of categories} and we call $(k,C,\omega)$ the object of $\TANN$ associated to $\M$.
 In the sequel, when given a model $\M=(k,B_p,X_p,B_{p,q},X_{p,q})$ of $\PROALG$ we will use this notation without further ado. For example, for $b\in B_p$ we will usually write $\omega(b)$ instead of $X_p(b)$.

  We also set $$G=G(\M)=\autt(\omega(\M)).$$ 
Combining Theorem \ref{theo: equivalence of categories} and Proposition \ref{prop: from TANN to PROALGEBRAIC GROUPS} we obtain:

\begin{cor} \label{cor: functor from PROALG to PRO}
	The functor $\M\rightsquigarrow G(\M)$ from the category of models of $\operatorname{PROALG}$ with homomorphism as the morphisms to the category $\PRO$ is full, essentially surjective and induces a bijection on the isomorphism classes. \qed
\end{cor}

\begin{ex} We describe the object $(k,C,\omega)$ of $\TANN$ that corresponds to the trivial proalgebraic group.
	The objects of $C$ are the zero object $\mathbf{0}$ together with all completely parenthesized finite sequences of integers greater or equal to $1$. For an object $b$ of $C$ corresponding to a complete parenthesization of the sequence $(n_1,\ldots,n_m)$ we set $\omega(b)=k^{n_1}\otimes_k\ldots\otimes_k k^{n_m}$. We also define $\omega(\mathbf{0})$ to be the zero vector space. For objects $b_1,b_2$ of $C$ the set of morphisms $\Hom(b_1,b_2)$ is defined as the set of $k$-linear maps from $\omega(b_1)$ to $\omega(b_2)$. The tensor product $\otimes\colon C\times C\to C$ is defined on non-zero objects as the concatenation of parenthesized sequences. We also set $b\otimes\mathbf{0}=\mathbf{0}\otimes b=\mathbf{0}$ for any object $b$ of $C$. On morphisms $\otimes$ is defines as the usual tensor product of $k$-linear maps.
\end{ex}

Let $G$ be a proalgebraic group. There does not seem to be a \emph{canonical} way to construct from $G$ a model $\M$ of $\PROALG$ such that $G(\M)\simeq G$. However, according to Corollary \ref{cor: functor from PROALG to PRO}, there always exists such a model $\M$. Moreover, if $\M_1$ and $\M_2$ are two models of $\PROALG$ such that $G(\M_1)$ and $G(\M_2)$ are isomorphic to $G$, then $\M_1$ and $\M_2$ are isomorphic. We can therefore safely define the theory $$\Th(G)$$ of $G$ as $\Th(\M)$ for any model $\M$ of $\PROALG$ such that $G(\M)\simeq G$. Two proalgebraic groups $G$ and $H$ (not necessarily living over the same base field) are \emph{elementarily equivalent} if $\Th(G)=\Th(H)$. We may also express this as $G\equiv H$. In a similar spirit, a class $\mathcal{C}$ of proalgebraic groups (potentially over varying base fields) is called \emph{elementary} if the class of all models of $\PROALG$ such that the associated proalgebraic group lies in $\mathcal{C}$ is elementary. 

We conclude this section with a discussion of some elementary classes of proalgebraic groups.

\begin{prop} \label{prop: elementary classes}
	The following classes of proalgebraic groups are elementary:
	\begin{itemize}
		\item The class of all diagonalizable proalgebraic groups.
		\item The class of all unipotent proalgebraic groups.
		\item The class of all linearly reductive proalgebraic groups.
	\end{itemize}
\end{prop}
\begin{proof}
	  The definition of diagonalizable proalgebraic groups is recalled in the beginning of the next section. Let us simply mention here that according to Prop. 1.6, Chapter IV, \S 1, in \cite{DemazureGabriel:GroupesAlgebriques} a proalgebraic group is diagonalizable if and only if every representation of $G$ is a direct sum of one-dimensional representations and the latter condition can be axiomatized. 
	
	A unipotent proalgebraic group can be defined as a proalgebraic group $G$ such that every representation of $G$ has a fixed vector (cf. \cite[Section 8.3]{Waterhouse:IntroductiontoAffineGroupSchemes}). This condition can be axiomatized by saying that for every representation $V$ there exists a morphism $\mathds{1}\to V$.
	
	Recall that a proalgebraic group $G$ is linearly reductive if and only if every representation of $G$ is a direct sum of irreducible representations. This condition can be axiomatized.		
	\end{proof}
In the following section we will show that the class of all algebraic groups is not elementary (Corollary~\ref{cor: algebraic not elementary}).

\section{Diagonalizable proalgebraic groups}

In this section we determine the theory of the multiplicative group $\Gm$ and deduce some basic consequences for the theory $\PROALG$ from this. We show that $\Th(\Gm)$ is determined by the theory of the base field and the theory of the abelian group $(\mathbb{Z},+)$. In fact, we establish a similar result for any diagonalizable proalgebraic group: The theory of a diagonalizable proalgebraic group $G$ is determined by the theory of the base field and the theory of the character group of $G$. Indeed, we show that the theory of all diagonalizable proalgebraic groups is weakly bi-interpretable with the two-sorted theory of pairs $(k,A)$, where $k$ is a field and $A$ an abelian group.

Let us first recall some basic facts about diagonalizable proalgebraic groups. See e.g., \cite[Section 2.2]{Waterhouse:IntroductiontoAffineGroupSchemes}, \cite[Section 12, c]{Milne:AlgebraicGroupsTheTheoryOfGroupSchemesOfFiniteTypeOverAField} or \cite[Chapter IV, \S 1, 1]{DemazureGabriel:GroupesAlgebriques}. Let $k$ be a field and let $A$ be an abelian group (usually written additively). The proalgebraic group $D(A)_k$ over $k$ is defined by $D(A)_k(R)=\Hom(A,R^\times)$ for any $k$-algebra $R$, where $\Hom(A,R^\times)$ denotes the abelian group of all morphisms of abelian groups from $A$ to $R^\times$. (Here, as usual, $R^\times$ denotes the multiplicative group of a ring $R$.) For example, $D(\mathbb{Z})_k\simeq \Gm$, or more generally, $D(\mathbb{Z}^n)_k\simeq \Gm^n$. A proalgebraic group is \emph{diagonalizable} if it is isomorphic to $D(A)_k$ for some abelian group $A$.
The functor $A\rightsquigarrow D(A)_k$ is an equivalence of categories from the category of abelian groups to the category of diagonalizable proalgebraic groups over $k$. The quasi-inverse is the functor that associates the character group to a diagonalizable proalgebraic group. Recall that the \emph{character group} $\chi(G)$ of a proalgebraic group $G$ is the abelian group of all morphisms of proalgebraic groups from $G$ to $\Gm$. Note that $\chi(G)$ is isomorphic to the abelian group of isomorphism classes of one-dimensional representations of $G$ under the tensor product.

As noted before, in general, for a proalgebraic group $G$, there does not seem to be a \emph{canonical} way to define a model $\M$ of $\PROALG$ such that $G(\M)\simeq G$. (Recall that by Corollary \ref{cor: functor from PROALG to PRO} such an $\M$ always exists and is unique up to an isomorphism.) However, if $G=D(A)_k$ is diagonalizable, there is a canonical choice which we will now describe. This will be helpful later on (Theorem \ref{theo: weak bi-interpretability for diagonalizable}) for showing that $\M$ is interpretable in the structure $(k,A)$.

Given a field $k$ and an abelian group $A$ we will now define a model $$\M(k,A)=(k(k,A),B_p(k,A),X_p(k,A),B_{p,q}(k,A),X_{p,q}(k,A))$$ of $\PROALG$ such that  $G(\M(k,A))\simeq D(A)_k$. We set $k(k,A)=k$ (including the field structure). The isomorphism classes of representations of $D(A)_k$ of dimension $n$ are in bijection with multisets of elements of $A$ of cardinality $n$ (see e.g., \cite[Prop. 1.6, Chapter IV, \S 1]{DemazureGabriel:GroupesAlgebriques}).
In more detail, for $a\in A$ let $k_a$ denote the one-dimensional representation of $D(A)_k$ given by the morphism $\chi_a\colon D(A)_k\to\Gm$ with $\chi_a(g)=g(a)$ for all $g\in D(A)_k(R)=\Hom(A,R^\times)$ and all $k$-algebras $R$. Then the map $$\{a_1,\ldots,a_n\}\mapsto W_{\{a_1,\ldots,a_n\}}=k_{a_1}\oplus\ldots\oplus k_{a_n}$$ induces a bijection between the set of multisets of cardinality $n$ of elements of $A$ and the set of isomorphism classes of $n$-dimensional representations of $D(A)_k$. So, for $n\geq 1$, we can define $B_{(1,n)}(k,A)$ as the set of multisets of elements from $A$ of cardinality $n$. In general, for $p=(m,n)$ we define $B_{p}(k,A)$ as the set of all completely parenthesized sequences of $m$-multisets $A_1,\ldots,A_m$ formed from elements from $A$ such that $|A_1|\ldots|A_m|=n$.
The representation corresponding to such a parenthesized sequence of multisets would be $W_{A_1}\otimes\ldots\otimes W_{A_m}$ with the corresponding parenthesization of the tensor product, where the $W_{A_i}$'s are understood to be tensor irreducible. Explicitly, for an element $b$ of $B_p(k,A)$ determining a parenthesization of $A_1,\ldots,A_m$, we let $\underline{v}_b$ denote the multiset $A_1\times \ldots\times A_m$, where an element $(a_1,\ldots,a_m)$ of $\underline{v}_b$ is considered as a parenthesized sequence with the same parenthesization pattern as the sequence $A_1\times \ldots\times A_m$. We let $V_b$ denote the $n$-dimensional $k$-vector space with basis $\underline{v}_b$. In other words, $V_b$ is the $k$-vector space of functions from $\underline{v}_b$ to $k$.

We define $X_p(k,A)$ to be the (disjoint) union of the $V_b$'s and we let $\pi_p\colon X_p(k,A)\to B_p(k,A)$ denote the map that maps an element in $V_b$ to $b$. We use the vector space structure on the $V_b$'s to define $0_p$, $A_p$, $LI_p$, and $SM_p$. We also define the interpretation of the constant symbol $1$ to correspond to the neutral element of $A$, considered as an element of $B_{(1,1)}(k,A)$.

We next want to define $\otimes_{p,\widehat{p}}$, where $p=(m,n)$ and $\widehat{p}=(\widehat{m},\widehat{n})$. Note that two elements $b\in B_p(k,A)$ and $\widehat{b}\in B_{\widehat{b}}(k,A)$ can be concatenated to an element $b\otimes \widehat {b}\in B_{p\widehat{p}}(k,A)$. Similarly, two elements $v\in \underline{v}_b$ and $\widehat{v}\in \underline{v}_{\widehat{b}}$ can be concatenated to an element $v\widehat{v}\in \underline{v}_{b\otimes\widehat{b}}$. This defines bilinear maps $V_b\times V_{\widehat{b}}\to V_{b\otimes \widehat{b}}$ that combine to a map $\otimes_{p,\widehat{p}}\colon X_p(k,A)\times X_{\widehat{p}}(k,A)\to X_{p\widehat{p}}(k,A)$.

We next want to define the morphism sorts. Note that for $a_1,a_2\in A$, there is a non-zero morphism (of $D(A)_k$-representations) from $k_{a_1}$ to $k_{a_2}$ if and only if $a_1=a_2$. Moreover, for $a\in A$, every linear map $k_a^{n_1}\to k_a^{n_2}$ is a morphism. This yields a description of the morphisms from $W_{A_1}=\oplus_{a_1\in A_1} k_{a_1}$ to $W_{A_2}=\oplus_{a_2\in A_2} k_{a_2}$ for finite multisets $A_1$ and $A_2$ consisting of elements of $A$: For $a\in A$ and $i=1,2$ let $$W_{i,a}=\bigoplus_{a_1\in A_1 \atop a_1=a}k_{a_1}$$
and let $\Hom_k(W_{1,a},W_{2,a})$ denote the set of $k$-linear maps from $W_{1,a}$ to $W_{2,a}$. 
Then the set of morphisms of $D(A)_k$-representations from $W_{A_1}$ to $W_{A_2}$ can be identified with $\prod_{a\in A}\Hom_k(W_{1,a},W_{2,a})$.

For $b\in B_p(k,A)$ and $v\in \underline{v}_b$ determining a parenthesization of $(a_1,\ldots,a_m)\in A^m$ we set $|v|=a_1+\ldots+a_m$. 
Furthermore, for $a\in A$ let $V_{b,a}$ denote the subspace of $V_b$ generated by all $v\in \underline{v}_b$ such that $|v|=a$. 

For $b\in B_p(k,A)$ and $\widehat{b}\in B_{\widehat{p}}(k,A)$ let 
$H_{b,\widehat{b}}$ denote  $\prod_{a\in A}\Hom_k(V_{b,a},V_{\widehat{b},a})$ considered as a subset of $\Hom_k(V_b,V_{\widehat{b}})$.
We set $$B_{p,\widehat{p}}(k,A)=\left\{(b,\widehat{b},H_{b,\widehat{b}})|\ b\in B_p(k,A),\ \widehat{b}\in B_{\widehat{p}}(k,A)\right\}$$
and 
$$X_{p,\widehat{p}}(k,A)=\left\{(b, \widehat{b}, h, v, h(v))|\ b\in B_p(k,A),\ \widehat{b}\in B_{\widehat{p}}(k,A),\ h\in H_{b,\widehat{b}},\ v\in V_b\right\}.$$
We define $\pi_{p,\widehat{p}}\colon X_{p,\widehat{p}}(k,A)\to B_{p,\widehat{p}}(k,A)$ to be the projection onto the first three factors while $S^B_{p,\widehat{p}}\colon B_{p,\widehat{p}}(k,A)\to B_p(k,A)$ and $T^B_{p,\widehat{p}}\colon B_{p,\widehat{p}}(k,A)\to B_{\widehat{p}}(k,A)$ are defined as the projections onto the first and second factor respectively. Similarly, $S^X_{p,\widehat{p}}\colon X_{p,\widehat{p}}(k,A)\to X_p(k,A)$ and $T^X_{p,\widehat{p}}\colon X_{p,\widehat{p}}(k,A)\to X_{\widehat{p}}(k,A)$ are defined as the projections onto the third and fourth factor respectively.

This completes our definition of the $\mathcal{L}$-structure $\M(k,A)$. It is clear from the construction (and the proof of Corollary \ref{cor: functor from PROALG to PRO}) that $\M(k,A)$ is a model of $\PROALG$ and that $G(\M(k,A))\simeq D(A)_k$.

\medskip

Since the addition in the character group can be described through the tensor product it is not surprising that the character group of $G(\M)$ is interpretable in $\M$ for any model $\M$ of $\PROALG$:

%
%

\begin{lemma} \label{lemma: interpret character group}
	Let $\M$ be a model of $\PROALG$. Then the character group of $G(\M)$ is definably interpreted in $\M$.
\end{lemma}
\begin{proof}
	Let $\M=(k,B_p,X_p,B_{p,q},X_{p,q})$. Then $B_{(1,1)}$ can be identified with the isomorphism classes of one-dimensional representations of $G=G(\M)$, i.e., with $\chi(G)$. The graph of addition in $\chi(G)=B_{(1,1)}$ consists of all $(b_1,b_2,b_3)\in B_{(1,1)}^3$ such that there exist $v_1\in \omega(b_1), v_2\in\omega(b_2)$ and $v_3\in \omega(b_3)$ and an isomorphism between $\pi_{(2,1)}(v_1\otimes v_2)$ and $b_3$. This set is $\emptyset$-definable. The identity element of $B_{(1,1)}$ is given by the constant symbol $1$.
\end{proof}

To proceed, let us recall the notion of \emph{interpretation} in the many-sorted context (cf. \cite[Chapter 5]{Hodges:ModelTheory} for the one-sorted case).  Let $\mathcal{L}$ and $\mathcal{L}'$ be many-sorted languages with sorts $S$ and $S'$ respectively. An \emph{interpretation $\Xi$ of $\mathcal{L}'$ in $\mathcal{L}$} is comprised of the following data:
\begin{itemize}
	\item For every sort $s'\in S'$ an $\mathcal{L}$-formula $\partial_{\Xi,s'}(x_1^{s_1},\ldots,x_n^{s_n})$ (the domain formula for the sort $s'$) in the free variables $\overline{x}_{s'}=(x_1^{s_1},\ldots,x_n^{s_n})$, where $n$ and $s_1,\ldots,s_n\in S$ depend on $s'$. (Here the notation $x^s$ means that the variable $x$ belongs to the sort $s$.)
	\item For every $s'\in S'$ an $\mathcal{L}$-formula $=_{\Xi, s'}(\overline{x}_{s'},\overline{y}_{s'})$ (the equivalence formula for the sort $s'$).
	\item For every function symbol $f'\in\mathcal{L'}$ that maps sorts $s'_1,\ldots,s'_r$ into sort $s'_{r+1}$ an $\mathcal{L}$-formula $f'_\Xi(\overline{x}_{s_1'},\ldots,\overline{x}_{s'_{r+1}})$.
	\item For every relation symbol $R'$ of $\mathcal{L}'$ between sorts $s'_1,\ldots,s'_r$ an $\mathcal{L}$-formula 
	$R'_{\Xi}(\overline{x}_{s_1'},\ldots,\overline{x}_{s'_{r}})$.
	\item For every constant symbol $c'$ in $\mathcal{L'}$ with sort $s'$ an $\mathcal{L}$-formula $c'_\Xi(\overline{x}_{s'})$.
\end{itemize} 
The \emph{admissibility conditions} of $\Xi$ are the $\mathcal{L}$-sentences expressing that for every $\mathcal{L}$-structure $M=(M_s)_{s\in S}$ the following holds:

\begin{itemize}
	\item For all $s'\in S'$ the formula $=_{\Xi,s'}$ defines an equivalence relation on $\partial_{\Xi,s'}(M)$. We will denote this equivalence relation simply by $\sim$ (even though it depends on $\Xi, s'$ and $M$). 
	\item For every function symbol $f'$ of ${\mathcal{L}'}$ that maps sorts $s'_1,\ldots,s'_r$ into sort $s'_{r+1}$ we have
	\begin{itemize}
		\item  $M\models f'_{\Xi}(\overline{a}_{s_1'},\ldots,\overline{a}_{s'_{r+1}})$ if and only if $M\models f'_{\Xi}(\overline{b}_{s_1'},\ldots,\overline{b}_{s'_{r+1}})$ for all tuples $\overline{a}_{s'_1},\ldots,\overline{a}_{s'_{r+1}},\overline{b}_{s'_1},\ldots,\overline{b}_{s'_{r+1}}$ from $M$ with $\overline{a}_{s'_1}\sim\overline{b}_{s'_1},\ldots, \overline{a}_{s'_{r+1}}\sim\overline{b}_{s'_{r+1}}$ and
		\item the induced subset of $\partial_{\Xi, s'_1}(M)/\sim\times\ldots\times \partial_{\Xi,s'_{r+1}}(M)/\sim$ is the graph of a function $f'(M)\colon \partial_{\Xi, s'_1}(M)/\sim\times\ldots\times\partial_{\Xi,s'_r}(M)/\sim\to \partial_{\Xi,s'_{r+1}}(M)/\sim$.
	\end{itemize}
\item  For every relation symbol $R'$ of $\mathcal{L}'$ between sorts $s'_1,\ldots,s'_r$ we have
 $M\models R'_{\Xi}(\overline{a}_{s_1'},\ldots,\overline{a}_{s'_{r}})$ if and only if $M\models R'_{\Xi}(\overline{b}_{s_1'},\ldots,\overline{b}_{s'_{r}})$ for all tuples $\overline{a}_{s'_1},\ldots,\overline{a}_{s'_{r}},\overline{b}_{s'_1},\ldots,\overline{b}_{s'_{r}}$ from $M$ with $\overline{a}_{s'_1}\sim\overline{b}_{s'_1},\ldots, \overline{a}_{s'_{r}}\sim\overline{b}_{s'_{r}}$. We thus have an induced relation $R'(M)$ on $\partial_{\Xi, s'_1}(M)/\sim\times\ldots\times\partial_{\Xi,s'_r}(M)/\sim$.
 \item 	For every constant symbol $c'$ in $\mathcal{L'}$ with sort $s'$ the realizations of $c'_{\Xi}(\overline{x}_{s'})$ in $M$ are an equivalence class $c'(M)$ in $\partial_{\Xi,s'}(M)$.
\end{itemize}
Note that if $M$ is an $\mathcal{L}$-structure that satisfies the admissibility conditions of $\Xi$, then
$\Xi(M)=(\partial_{\Xi,s'}(M)/\sim)_{s'\in S'}$ is an $\mathcal{L}'$-structure.

Now let $T$ be an $\mathcal{L}$-theory and let $T'$ be an $\mathcal{L}'$-theory. An interpretation $\Xi$ of $\mathcal{L}'$ in $\mathcal{L}$ is a \emph{left total interpretation} of $T'$ in $T$ if every model of $T$ satisfies the admissibility conditions of $\Xi$ and if for every model $M$ of $T$ the $\mathcal{L}'$-structure $\Xi(M)$ is a model of $T'$. Thus, from every model $M$ of $T$ we get a model of $\Xi(M)$ of $T'$.

Finally, following \cite[Section 3.3.]{Visser:CategoriesOfTheoriesAndInterpretations}, the theories $T$ and $T'$ are \emph{weakly bi-interpretable} if there exists a left total interpretation $\Xi$ of $T'$ in $T$ and a left total interpretation $\Omega$ of $T$ in $T'$ such that for every model $M$ of $T$ the $\mathcal{L}$-structures $\Omega(\Xi(M)$ and $M$ are isomorphic and for every model $M'$ of $T'$ the $\mathcal{L}'$-structures $\Xi(\Omega(M'))$ and $M'$ are isomorphic. (Bi-interpretability is the slightly stronger notion where the above isomorphisms are required to be uniformly definable.)

By the \emph{theory of diagonalizable proalgebraic groups} we mean the set of all $\mathcal{L}$-sentences true in all models $\M$ of $\PROALG$ such that $G(\M)$ is diagonalizable (cf. Proposition \ref{prop: elementary classes}).

\begin{theo} \label{theo: weak bi-interpretability for diagonalizable}
	The theory of diagonalizable proalgebraic groups is weakly bi-interpretable with the two-sorted theory of pairs $(k,A)$, where $k$ is a field and $A$ an abelian group.
\end{theo}
\begin{proof}
	Let $\mathcal{L}'$ denote the two-sorted language with the language of rings on the first sort (the field sort) and the language of abelian groups on the second sort (the group sort). Let $T'$ denote the $\mathcal{L}'$-theory of all pairs $(k,A)$ where $k$ is a field and $A$ an abelian group. Moreover, let $T$ denote the $\mathcal{L}$-theory of diagonalizable proalgebraic groups.

	The interpretation $\Xi$ of $\mathcal{L}'$ in $\mathcal{L}$ is relatively easy to describe (cf. Lemma \ref{lemma: interpret character group}): The domain formula for the field sort of $\mathcal{L}'$ is trivial (i.e., equal to $x_1=x_1$, where $x_1$ belongs to the field sort of $\mathcal{L}$) so that it returns the field sort of $\mathcal{L}$. Similarly, the domain formula for the group sort of $\mathcal{L}'$ is trivial so that it returns the sort $B_{(1,1)}$ of $\mathcal{L}$. The two equivalence formulas are also trivial, so that the corresponding equivalence relation simply expresses equality of elements. 
	The interpretation of the ring language on the field sort of $\mathcal{L}'$ is the ring language on the field sort of $\mathcal{L}$.
	The $\mathcal{L}'$-symbol for the identity element of the group is to be interpreted as the $\mathcal{L}$-symbol $1$ (corresponding to the trivial representation). Finally, the addition symbol $+$ on the group sort, yields the formula $+_\Xi$ that defines in every model $\M$ of $T$ the set of all
	 $(b_1,b_2,b_3)\in B_{(1,1)}^3$ such that there exist $v_1\in \omega(b_1), v_2\in\omega(b_2)$ and $v_3\in \omega(b_3)$ and an isomorphism between $\pi_{(2,1)}(v_1\otimes v_2)$ and $b_3$. Clearly $\Xi$ is a left total interpretation of $T'$ in $T$.

	\medskip

	 We will next construct an interpretation $\Omega$ of $\mathcal{L}$ in $\mathcal{L}'$. The idea for the construction is rather simple but a little tedious to implement. The formulas for $\Omega$ boil down to interpreting the $\mathcal{L}$-structure $\M(k,A)$ defined above in the $\mathcal{L}'$-structure $(k,A)$. 	 
 We begin with the domain formulas $\partial_{\Omega,s}$, where $s$ is a sort from $\mathcal{L}$. The domain formula for the field sort of $\mathcal{L}$ simply returns the field sort of $\mathcal{L}'$.

{\bf Definition of $\partial_{\Omega,B_p}$:} For $p=(m,n)$, we consider, for every $\mathcal{L}'$-structure $(k,A)$, the set $P_p(k,A)$ of completely parenthesized sequences  
	$$(a_1^1,\ldots,a^1_{n_1}),\ldots,(a_{1}^m,\ldots,a_{n_m}^m)$$ of sequences in $A$ with $n_1\ldots n_m=n$.
To describe $P_p(k,A)$ inside $(k,A)$ we can encode the pattern associated with such a parenthesization of a sequence of sequences in a sequence of zero's and one's belonging to $k$. While there are different ways to do this, for the sake of concreteness, let us fix the following decoding. A sequence of zero's and one's always has to be read by blocks of two elements according to the following convention:
	\begin{itemize}
		\item A block $10$ is to be read as an opening parenthesis ``(''.
		\item A block $01$ is to be read as a closing parenthesis ``)''.
		\item A block $00$ is to be interpreted as a place holder for an element of $A$.
	\end{itemize}
We also use parenthesis to separate place holders that correspond to different sequences in $A$.
For example, the element $\Big(\big(( a_1,a_2,a_3) (a_4)\big) (a_5,a_6)\Big)$ of $P_{(3,6)}(k,A)$
yields the pattern
$$(((\bullet \bullet \bullet)(\bullet))(\bullet \bullet))$$
 that is encoded in the sequence $$ 10 \ 10 \ 10 \ 00 \ 00 \ 00 \ 01 \ 10 \ 00 \ 01 \ 01 \ 10 \ 00 \ 00 \ 01 \ 01.$$
Note that different patterns may yield binary sequences of different lengths. Let $r=r(p)$ denote the maximal length of all theses binary sequences and let $D_{p}(k,A)\subseteq \{0,1\}^r\subseteq k^r$ denote the set of all binary sequences that are derived from elements of $P_p(k,A)$. Here a binary sequence of length less than $r$ is extended to a sequence of length $r$ by adding $11$-blocks. Let $s=s(p)$ denote the maximum number of $00$-blocks that occur in any element of $D_{p}(k,A)$ and let $F_{p}(k,A)$ denote the subset of $k^r\times A^s$ consisting of all elements of the form $(d,a_1,\ldots,a_t,0,\ldots,0)$, where $d\in D_{p}(k,A)$, $t$ is the number of $00$-blocks occurring in $d$ and $a_1,\ldots,a_t\in A$. By construction, the set $F_p(k,A)$ is in bijection with $P_p(k,A)$.

We let $\partial_{\Omega,B_{p}}$ denote the $\mathcal{L}'$-formula that defines in every $\mathcal{L}'$-structure $(k, A)$ the subset $F_{p}(k,A)$ of $k^r\times A^s$. An element $f$ of $F_p(k,A)$ corresponding to a complete parenthesization of a sequence $$(a_1^1,\ldots,a^1_{n_1}),\ldots,(a_{1}^m,\ldots,a_{n_m}^m)$$ of sequences in $A$ with $n_1\ldots n_m=n$ yields a totally ordered multiset $\underline{v}_f$ of cardinality $n$ consisting of all sequences of length $m$ in $A$ that are of the form $(a^1_{i_1},\ldots,a^m_{i_m})$ with $1\leq i_1\leq n_1,\ldots,1\leq i_m\leq n_m$. Alternatively, we can define $\underline{v}_f$ as the multiset product $\{a_1^1,\ldots,a_{n_1}^1\}\times\ldots\times \{a_1^m,\ldots,a_{n_m}^m\}$. The order on $\underline{v}_f$ is obtained by stipulating that $a_1^j<a_2^j<\ldots<a_{n_j}^j$ for $j=1,\ldots,m$ and then using the lexicographic order.
Let $V_f$ denote the $k$-vector space with basis $\underline{v}_f$. We think of an element $f$ of $F_p(k,A)$ as determining the pair $(V_f,\underline{v}_f)$, where $\underline{v}_f$ is an ordered basis of $V_f$.

\medskip

{\bf Definition of $\partial_{\Omega,X_p}$:} 
We let $\partial_{\Omega,X_p}$ denote the formula that defines the set $F_p(k,A)\times k^n$ in every $\mathcal{L}'$-structure $(k,A)$. We think of an element $(f,\xi)\in F_p(k,A)\times k^n$ as determining the element $\underline{v}_f\xi=\xi_1 v_{f,1}+\ldots +\xi_n v_{f,n}$ of $V_f$, where $\xi=(\xi_1,\ldots,\xi_n)$ and $\underline{v}_f=\{v_{f,1},\ldots,v_{f,n}\}$ with $v_{f,1}<v_{f,2}<\ldots<v_{f,n}$.

\medskip

{\bf Definition of $\partial_{\Omega,B_{p,\widehat{p}}}$:}
Let $p=(m,n)$ and $\widehat{p}=(\widehat{m},\widehat{n})$. For every $\mathcal{L}'$-structure $(k,A)$ and $(a_1,\ldots,a_m)\in A^m$ we set $|a|=a_1+\ldots+a_m$ for a fixed but arbitrary parenthesization of this sum\footnote{The  parenthesization will ultimately not matter since only the case when $A$ is an abelian group is relevant for us.}. For $f\in F_p(k,A)$ we let $\Sigma(f)=\{|v|\ |\ v \in\underline{v}_f\}$ denote the multiset of all sums of all tuples in $\underline{v}_f$. Furthermore, for $a\in A$ we let $m_f(a)$ denote the multiplicity of $a$ in $\Sigma(f)$. Of course $m_f(a)=0$ for all but finitely many $a\in A$.
For $(f,\widehat{f})\in F_p(k,A)\times F_{\widehat{p}}(k,A)$ we set $r(f,\widehat{f})=\sum_{a\in A}m_f(a)m_{\widehat{f}}(a)$. Let $$r=\max \left\{r(f,\widehat{f})\Big| \ (f,\widehat{f})\in F_p(k,A)\times F_{\widehat{p}}(k,A)\right\}.$$
We define $H_{p,\widehat{p}}(k,A)$ to be the subset of $F_p(k,A)\times F_{\widehat{p}}(k,A)\times k^r$ consisting of all elements of the form $(f,\widehat{f},\lambda)$, where $f\in  F_{\widehat{p}}(k,A)$, $\widehat{f}\in F_{\widehat{p}}(k,A)$, $\lambda=(\lambda_1,\ldots,\lambda_r)\in k^r$ and $\lambda_i=0$ for $i>r(f,\widehat{f})$.
 Let $\partial_{\Omega,B_{p,\widehat{p}}}$ denote the $\mathcal{L}'$-formula that defines in every $\mathcal{L}'$-structure $(k,A)$ the set $H_{p,\widehat{p}}(k,A)$.

  We think of an element $(f,\widehat{f},\lambda)$ of $H_{p,\widehat{p}}(k,A)$ as defining a morphism $\psi_{(f,\widehat{f},\lambda)}\colon (V_f,\underline{v}_f)\to (V_{\widehat{f}},\underline{v}_{\widehat{f}})$ as follows:
 To simplify the formulas we set $m(a)=m_f(a)$ and $\widehat{m}(a)=m_{\widehat{f}}(a)$.
For each of the finitely many $a\in A$ such that $m(a)\widehat{m}(a)\geq 1$ let $$I_a=\{i_{1,a},\ldots,i_{m(a),a}\}=\{ i\in\{1,\ldots,n\}|\ |v_{f,i}|=a\}$$ and $$J_a=\{j_{1,a},\ldots,j_{\widehat{m}(a),a}\}=\{ j\in\{1,\ldots,\widehat{n}\}|\ |v_{\widehat{f},j}|=a\}$$ with 
$v_{f,i_{1,a}}<\ldots<v_{f,{i_{m(a),a}}}$ and $v_{\widehat{f},j_{1,a}}<\ldots<v_{\widehat{f},j_{\widehat{m}(a),a}}$.  We order the sets of the form $I_a$ by comparing the $v_{f,i_{1,a}}$. Say $I(a_1)<\ldots<I(a_s)$. For every $\ell=1,\ldots,s$ let $\underline{v}_{f,I_{a_\ell}}$ denote the (ordered) sequence of elements of $\underline{v}_f$ whose index belongs to $I_{a_\ell}$. We define $v_{\widehat{f},J_{a_\ell}}$ similarly. We define a $k$-linear map $\psi_\ell$ from the subspace of $V_f$ generated by $\underline{v}_{f,I_{a_\ell}}$ to the subspace of $V_{\widehat{f}}$ generated by $\underline{v}_{\widehat{f},J_{a_\ell}}$ by setting $\psi_\ell(\underline{v}_{f,I_{a_\ell}})=\underline{v}_{\widehat{f},J_{a_\ell}} M_{a_\ell}$, where $M_{a_\ell}$ is the $m(a_\ell)\times \widehat{m}(a_\ell)$-matrix obtained by putting, row by row, the entries of $\lambda=(\lambda_1,\ldots,\lambda_r)$ that start at index $m(a_1)\widehat{m}(a_1)+\ldots+m(a_{l-1})\widehat{m}(a_{l-1})+1$ and end at index $m(a_1)\widehat{m}(a_1)+\ldots+m(a_{l})\widehat{m}(a_l)$ into a matrix. Finally, we define the linear map 
$\psi_{(f,\widehat{f},\lambda)}\colon V_f\to V_{\widehat{f}}$ by $$\psi_{(f,\widehat{f},\lambda)}(v_{f,i})=\begin{cases}
\psi_\ell(v_{f,i}) \ \text{ if } i\in I_{a_\ell} \\
0 \quad \quad \text{ otherwise. }
\end{cases}$$ 

\medskip

{\bf Definition of $\partial_{\Omega,X_{p,\widehat{p}}}$:}
We define\footnote{Again the implicit use of matrix multiplication in this definition need not concern us, since ultimately we are only interested in the case when $k$ is a field and $A$ an abelian group.} $\partial_{\Omega,X_{p,\widehat{p}}}$ to be the $\mathcal{L}'$-formula that defines in every $\mathcal{L}'$-structure $(k,A)$ the set of all $(f,\widehat{f},\lambda,\xi,\widehat{\xi})$ such that $(f,\widehat{f},\lambda)\in H_{p,\widehat{p}}(k,A)$, $(f,\xi)\in\partial_{\Omega,X_p}(k,A)$,  $(\widehat{f},\widehat{\xi})\in\partial_{\Omega,X_{\widehat{p}}}(k,A)$ 
 and $\psi_{(f,\widehat{f},\lambda)}(\underline{v}_f\xi)=\underline{v}_{\widehat{f}}\widehat{\xi}$.
This concludes the definition of the domain formulas $\partial_{\Omega,s}$ for all sorts $s$ of $\mathcal{L}$.

\medskip
 We will next define the equivalence formulas $=_{\Omega,s}$. We define $=_{\Omega,k}$ to be the $\mathcal{L}'$-formula $x_1=x_2$, where $x_1$ and $x_2$ are variables from the field sort. (So the equivalence relation on the field sort is trivial.)

\medskip

{\bf Definition of $=_{\Omega,B_{p}}$:} We let $=_{\Omega,B_{p}}$ denote the $\mathcal{L}'$-formula that defines in every $\mathcal{L}'$-structure $(k,A)$ the following equivalence relation on $F_{p}(k,A)$: For $(d,a_1,\ldots,a_t,0,\ldots,0), (d',a'_1,\ldots,a'_{t'},0,\ldots,0)\in F_p(k,A)$ with $d,d'\in D_p(k,A)$ and $a_1,\ldots,a_t,a_1',\ldots,a_{t'}'\in A$ we have $$(d,a_1,\ldots,a_t,0,\ldots,0)\sim (d',a'_1,\ldots,a'_{t'},0,\ldots,0)$$
if $d=d'$ (so that also $t=t'$) and $a'_1,\ldots,a'_{t'}$ is a permutation of $a_1,\ldots,a_t$, where elements corresponding to the same string of $00$-blocks in $d=d'$ are permuted among themselves. Note that the map $f\mapsto b(f)$ that assigns to an $f\in F_p(k,A)$ corresponding to a complete parenthesization of a sequence $$(a_1^1,\ldots,a^1_{n_1}),\ldots,(a_{1}^m,\ldots,a_{n_m}^m)$$ of sequences in $A$ the corresponding parenthesization of the sequence $$\{a_1^1,\ldots,a^1_{n_1}\},\ldots,\{a_{1}^m,\ldots,a_{n_m}^m\} $$ of multisets, induces a bijection between $\partial_{\Omega, B_{p}}(k,A)/\sim$ and $B_p(k,A)$.

%
We note that $V_f$ only depends on the equivalence class of $f\in F_p(k,A)$. Indeed, the multiset underlying $\underline{v}_f$ only depends on the equivalence class of $f$. Only the ordering on the multiset $\underline{v}_f$ depends on $f$. Moreover, if the equivalence class of $f$ maps to $b\in B_p(k,A)$ under
$\partial_{\Omega, B_{p}}(k,A)/\sim\simeq B_p(k,A)$, then, with the notation introduced in the beginning of this section, $V_f=V_b$. 

\medskip

{\bf Definition of  $=_{\Omega,X_p}$:}
 Let $=_{\Omega,X_p}$ denote the formula that defines on every $\partial_{\Omega,X_p}(k,A)=F_p(k,A)\times k^n$ the equivalence relation 
$$(f,\xi)\sim (f',\xi') \Leftrightarrow f\sim f' \text{ and } \xi' \text{ is a permutation of } \xi \text{ such that } \underline{v}_f\xi=\underline{v}_{f'}\xi'.$$
We note that the map $(f,\xi)\to \underline{v}_f\xi$ induces a bijection $\partial_{\Omega,X_p}(k,A)/\sim\to \uplus V_f$, where the disjoint union is taken over all equivalence classes in $F_p(k,A)$. In other words, $\partial_{\Omega,X_p}(k,A)/\sim$ is in bijection with $X_p(k,A)$.

\medskip

{\bf Definition of $=_{\Omega,B_{p,\widehat{p}}}$:} 
Let $=_{\Omega,B_{p,\widehat{p}}}$ denote the $\mathcal{L}'$-formula that defines on every $H_{p,\widehat{p}}(k,A)$ the equivalence relation 
$$(f_1,\widehat{f_1},\lambda_1)\sim (f_2,\widehat{f_2},\lambda_2) \ \Leftrightarrow \ f_1\sim f_2,\ \widehat{f_1}\sim\widehat{f_2} \text{ and $\lambda_2$ is a permutation of $\lambda_1$ such that } \psi_{(f_1,\widehat{f_1},\lambda_1)}=\psi_{(f_2,\widehat{f_2},\lambda_2)}.$$
Then the map $(f,\widehat{f},\lambda)\to(b(f), b(\widehat{f}), \psi_{(f,\widehat{f},\lambda)})$ induces a bijection between $\partial_{\Omega,B_{p,\widehat{p}}}(k,A)/\sim$ and $B_{p,\widehat{p}}(k,A)$.

\medskip

{\bf Definition of  $=_{\Omega,X_{p,\widehat{p}}}$:}
Let $=_{\Omega,X_{p,\widehat{p}}}$ denote the $\mathcal{L}'$-formula that defines on every $\partial_{\Omega,X_{p,\widehat{p}}}(k,A)$ the equivalence relation
$$(f_1,\widehat{f}_1,\lambda_1,\xi_1,\widehat{\xi}_1)\sim (f_2,\widehat{f}_2,\lambda_2,\xi_2,\widehat{\xi}_2)\ \Leftrightarrow (f_1,\widehat{f_1},\lambda_1)\sim (f_2,\widehat{f_2},\lambda_2), \ (f_1,\xi_1)\sim (f_2,\xi_2) \text{ and } (\widehat{f}_1,\widehat{\xi}_1)\sim (\widehat{f}_2,\widehat{\xi}_2).$$
Then the map $(f,\widehat{f},\lambda,\xi,\widehat{\xi})\mapsto (b(f),b(\widehat{f}),\psi_{(f,\widehat{f},\lambda)},\underline{v}_f\xi, \underline{v}_{\widehat{f}}\widehat{\xi})$ induces a bijection between $\partial_{\Omega,X_{p,\widehat{p}}}(k,A)/\sim$ and $X_{p,\widehat{p}}(k,A)$.

\medskip

This concludes the definition of the equivalence formulas for $\Omega$. Note that for every model $(k,A)$ of $T'$ we have a bijection between $(\partial_{\Omega,s}(k,A)/\sim)_{s\in S}$ (where $S$ denotes the set of sorts for $\mathcal{L}$) and $\M(k,A)$. Using this bijection, the interpretation of the symbols of $\mathcal{L}$ in $\M(k,A)$ gives rise to an interpretation of the symbols of $\mathcal{L}$ in $(\partial_{\Omega,s}(k,A)/\sim)_{s\in S}$. It is now straight forward to check that these interpretations can be defined (uniformly in $(k,A)$) by appropriate  $\mathcal{L}'$-formulas. This completes the definition of $\Omega$. Note that $\Omega(k,A)\simeq \M(k,A)$ for every model $(k,A)$ of $T'$.

It is clear that $\Omega$ is a left total interpretation of $T$ in $T'$. Moreover, $\Xi(\Omega(k,A))\simeq (k,A)$ for every model $(k,A)$ of $T'$.

For a model $\M=(k,B_p,X_p,B_{p,q},X_{p,q})$ of $T$, let us consider $B_{(1,1)}$ as an abelian group (via the identification of $B_{(1,1)}$ with the character group of $G(\M)$ as in the definition of $\Xi$). Then $\M\simeq \M(k,B_{(1,1)})$ because $G(\M)$ and $G(\M(k,B_{(1,1)}))$ are both isomorphic to $D(B_{(1,1)})_k$.
Moreover, $\Omega(\Xi(\M))=\Omega(k,B_{(1,1)})\simeq \M(k,B_{(1,1)})$. Thus $\Omega(\Xi(\M))\simeq \M$ as desired.
%
%
\end{proof}

We note that the above isomorphism $\Omega(\Xi(\M))\simeq \M$ is not canonical. For example, on the $p$-total objects sort we need a bijection between $\partial_{\Omega,X_p}(k,B_{(1,1)})/\sim=(F_p(k,B_{(1,1)})\times k^n)/
\sim$ and $X_p$. Specifying such a bijection involves the choice of appropriate bases. This is why we have weak bi-interpretability rather than bi-interpretability in Theorem \ref{theo: weak bi-interpretability for diagonalizable}.

In the course of the proof of Theorem \ref{theo: weak bi-interpretability for diagonalizable} we have seen the following:
\begin{cor} \label{cor: interpret M in A}
	Let $\M$ be a model of $\PROALG$ such that $G(\M)$ is diagonalizable. Then $\M$ is interpretable in $(k,A)$, where $k$ is the field sort of $\M$ and $A$ the character group of $G(\M)$. \qed
\end{cor}

Based on Theorem \ref{theo: weak bi-interpretability for diagonalizable} we can now characterize elementary equivalence and elementary extensions for diagonalizable proalgebraic groups.

\begin{cor}
	Let $k$ be a field and $G$ a diagonalizable proalgebraic group over $k$. Then a proalgebraic group $G'$ over a field $k'$ is elementarily equivalent to $G$ if and only if $k'$ is elementarily equivalent to $k$, $G'$ is diagonalizable and $\chi(G')$ is elementarily equivalent to $\chi(G)$.
\end{cor}
\begin{proof}
	First assume that $G'\equiv G$. Then clearly $k'\equiv k$. Moreover, we know from Proposition \ref{prop: elementary classes} that $G'$ must be diagonalizable. Let $\M=(k,B_p,X_p,B_{p,q},X_{p,q})$ and $\M'=(k',B'_p,X'_p,B'_{p,q},X'_{p,q})$ be models of $\PROALG$ such that $G(\M)\simeq G$ and $G(\M')\simeq G'$. Since interpretations preserve elementary equivalence we see that $\Xi(\M')\equiv\Xi(\M)$, where $\Xi$ is defined as in the proof of Theorem~\ref{theo: weak bi-interpretability for diagonalizable}. So $(k',B'_{(1,1)})\equiv (k,B_{(1,1)})$. Since $B'_{(1,1)}$ and $B_{(1,1)}$ are isomorphic with $\chi(G')$ and $\chi(G)$ respectively, we see that $\chi(G)\equiv\chi(G')$.
	
	Conversely, assume now that $G'$ is diagonalizable, $k'\equiv k$ and $\chi(G')\equiv \chi(G)$. Then also $(k',\chi(G'))\equiv (k,\chi(G))$. Therefore $\Omega(k',\chi(G'))\equiv\Omega(k,\chi(G))$, where $\Omega$ is defined as in the proof of Theorem \ref{theo: weak bi-interpretability for diagonalizable}. But $\M'\simeq \Omega(k',\chi(G'))$ and $\M\simeq \Omega(k,\chi(G))$. Thus $\M'\equiv \M$, i.e., $G'\equiv G$.
\end{proof}

In particular, for a field $k$, a proalgebraic group $G$ over $k$ is elementarily equivalent to the multiplicative group $\Gm$ over $k$ if and only if $G$ is isomorphic to $D(A)_k$ and $A$ is elementarily equivalent to $(\mathbb{Z}, +)$. Since $\Th(\mathbb{Z},+)$ has models that are not finitely generated as abelian groups (e.g., $\mathbb{Z}\oplus\mathbb{Q}$, see \cite{EklofFischer:TheElementaryTheoryOfAbelianGroups}) and $D(A)_k$ is algebraic if and only if $A$ is finitely generated we obtain:

\begin{cor} \label{cor: algebraic not elementary}
	The class of all algebraic groups is not elementary. \qed
\end{cor}

\begin{cor}
	Let $\M$ be a model of $\PROALG$ such that $G(\M)$ is diagonalizable and let $A$ denote the character group of $G(\M)$. If $\M' \succeq\M$ is an elementary extension of $\M$, then $G(\M')\simeq D(A')_{k'}$, where $k'\succeq k$ and $A'\succeq A$.
	
	Conversely, if $k'\succeq k$ and $A'\succeq A$ are elementary extensions, then there exists an elementary extension $\M'\succeq \M$ such that $G(\M')\simeq D(A')_{k'}$.
\end{cor}
\begin{proof}
	Again, let $\Xi$ and $\Omega$ be the interpretations defined in the proof of Theorem \ref{theo: weak bi-interpretability for diagonalizable}. We observe that $G(\M')$ is diagonalizable by Proposition \ref{prop: elementary classes}. Let $A'$ denote the character group of $G(\M')$. Since interpretations preserve elementary embeddings (cf. \cite[Theorem 5.3.4]{Hodges:ModelTheory} for the one-sorted case) we see that $\Xi(\M')\succeq \Xi(\M)$, i.e., $(k',A')\succeq (k,A)$. It follows that $k'\succeq k$ and $A'\succeq A$.

	Conversely, assume we start with elementary extensions $k'\succeq k$ and $A'\succeq A$. Then $(k',A')\succeq (k,A)$ and $\Omega(k',A')\succeq\Omega(k,A)$. As $\Omega(k,A)\simeq \M$ and $G(\Omega(k',A'))\simeq D(A')_{k'}$ the claim follows.	
\end{proof}

\begin{cor}
	Let $G$ be a diagonalizable proalgebraic group over an algebraically closed field $k$. Then $\Th(G)$ is stable, but not necessarily superstable. 
\end{cor}
\begin{proof}
	Let $\M$ be a model of $\PROALG$ such that $G(\M)\simeq G$ and let $A$ denote the character group of $G(\M)$. Since $\Th(k)$ is stable and $\Th(A)$ is stable, it follows that also $\Th(k,A)$ is stable. As $\M$ can be interpreted in $(k,A)$ by Corollary \ref{cor: interpret M in A} it follows that $\Th(\M)$ is stable.
	
	There are abelian groups whose theory is not superstable (e.g., an infinite direct sum of copies of $\mathbb{Z}$, see \cite[Theorem A.2.13]{Hodges:ModelTheory}). Since these can be interpreted in a model $\M$ of $\PROALG$ with $G(\M)$ diagonalizable, it follows that $\Th(\M)$ cannot be superstable.
\end{proof}


\section{Types}

We postpone a more comprehensive study of types for certain models of $\PROALG$ to a future publication. Here we establish some initial algebraic results that illustrate the expressive power of $\PROALG$: 
\begin{itemize}
	\item For a model $\M=(k, B_p, X_p, B_{p,q}, X_{p,q})$ of $\PROALG$ the type of an element $b\in B_p$ over the empty set determines the minimal degree of defining equations of the image of the representation $G(\M)\to \Gl_{\omega(b)}$ associated with $b$. 
	\item The type of $b$ over $k$ determines the image of $G(\M)\to \Gl_{\omega(b)}$.
\end{itemize}

The key to these results is the fact that the type of $b$ over $k$ knows which subspaces of representations of $G(\M)$ obtained from $\omega(b)$ by forming tensor products, duals and direct sums are $G(\M)$-stable (i.e., subrepresentations). Moreover, the image of $G(\M)\to \Gl_{\omega(b)}$ is determined by these stable subspaces.

\subsection{Stable subspaces and defining equations}

The results in this subsection are of a preparatory nature and purely algebraic, i.e., do not involve any model theory. It is well known that any closed subgroup $G$ of $\Gl_V$, for a finite dimensional vector space $V$, is the stabilizer of some subspace of a representation of $\Gl_V$ obtained from $V$ by forming tensor products, duals and direct sums. We will need to understand this result in full detail. In particular, we would like to know how the degree of the polynomials defining the stabilizer is related to the constructions (tensor product, duals, direct sums) applied to $V$. 

\medskip

Let $G$ be a closed subgroup of $\Gl_n$. Then the defining ideal $\I(G)$ of $G$ is a Hopf ideal of the Hopf algebra $k[\Gl_n]=k[Z,1/\det(Z)]=k[Z,Z^{-1}]$, where $Z$ is an $n\times n$ matrix of indeterminates. For every $d\geq 0$ let $k[Z,Z^{-1}]_{\leq d}$ denote the finite dimensional $k$-subspace of $k[Z,Z^{-1}]$ consisting of all elements of the form $P(Z,Z^{-1})$, where $P$ is a polynomial over $k$ in $2n^2$ variables of degree at most $d$.

\begin{lemma} \label{lemma: is subcoalgebra} 
	The subspace $k[Z,Z^{-1}]_{\leq d}$ is a subcoalgebra of $k[Z,Z^{-1}]$.
\end{lemma}
\begin{proof}
	Let $\Delta\colon k[Z,Z^{-1}]\to k[Z,Z^{-1}]\otimes_k k[Z,Z^{-1}],\ Z\mapsto Z\otimes Z$ denote the comultiplication. Here $Z\otimes Z$ is the $n\times n$ matrix whose $ij$-entry is $\sum_{\ell=1}^nZ_{i\ell}\otimes Z_{\ell j}$. In other words, $Z\otimes Z$ is the (matrix) product of the  matrices $Z\otimes 1$ and $1\otimes Z$, where $Z\otimes 1=(Z_{i,j}\otimes 1)_{1\leq i,j\leq n}\in   (k[Z,Z^{-1}]\otimes_k k[Z,Z^{-1}])^{n\times n}$ and $1\otimes Z=(1\otimes Z_{i,j})_{1\leq i,j\leq n}\in   (k[Z,Z^{-1}]\otimes_k k[Z,Z^{-1}])^{n\times n}$. We have $\Delta(Z^{-1})=\Delta(Z)^{-1}=((Z\otimes 1) (1\otimes Z))^{-1}=(1\otimes Z)^{-1}(Z\otimes 1)^{-1}=(1\otimes Z^{-1}) (Z^{-1}\otimes 1)$. So $(\Delta(Z^{-1}))_{ij}=\sum_{\ell=1}^{n}(Z^{-1})_{\ell j}\otimes (Z^{-1})_{i\ell}$. Consequently, if $P$ is a polynomial of degree at most $d$, then $\Delta(P(Z,Z^{-1}))=P(\Delta(Z),\Delta(Z^{-1}))\in k[Z,Z^{-1}]_{\leq d}\otimes_k k[Z,Z^{-1}]_{\leq d}$.
\end{proof}

\begin{lemma} \label{lemma: is Hopf ideal}
	The ideal of $k[\Gl_n]=k[Z,Z^{-1}]$ generated by $\I(G)\cap k[Z,Z^{-1}]_{\leq d}$ is a Hopf ideal. 
\end{lemma}
\begin{proof}
	If $C$ is a coalgebra with a coideal $V$ and a subcoalgebra $D$, then $V\cap D$ is a coideal of $D$, and so, in particular, a coideal of $C$. (To see this note that $D\to C/V$ is a morphism of coalgebras with kernel $V\cap D$ and kernels of morphisms of coalgebras are coideals.) It follows, using Lemma~\ref{lemma: is subcoalgebra}, that $\I(G)\cap k[Z,Z^{-1}]_{\leq d}$ is a coideal of $k[Z,Z^{-1}]$.

	Let $I=(\I(G)\cap k[Z,Z^{-1}]_{\leq d})$ denote the ideal of $k[Z,Z^{-1}]$ generated by $\I(G)\cap k[Z,Z^{-1}]_{\leq d}$. In any commutative Hopf algebra, the ideal generated by a coideal is a coideal. It follows that $I$ is a coideal.

	Thus it only remains to check that $I$ is stable under the antipode $S\colon k[Z,Z^{-1}]\to k[Z,Z^{-1}],\ Z\mapsto Z^{-1}$. However, since $k[Z,Z^{-1}]_{\leq d}$ and $\I(G)$ are stable under $S$ this is immediate.
\end{proof}

Let $V$ be a finite dimensional $k$-vector space. The choice of a basis of $V$ allows us to identify $\Gl_V$ with $\Gl_n$ and $k[\Gl_V]$ with $k[Z,Z^{-1}]$. We note that $k[Z,Z^{-1}]_{\leq d}$, considered as a subspace of $k[\Gl_V]$, does not depend on the choice of the basis and we may therefore safely denote it by $k[\Gl_V]_{\leq d}$.

The following notation will be useful: For a polynomial $P\in\mathbb{N}[X,Y]$ in two variables and a finite dimensional $k$-vector space $V$ we define $P(V,V^\vee)$ as the $k$-vector space obtained from $V$ and $P$ by replacing $X$ by $V$, $Y$ by the dual vector space $V^\vee$ of $V$, addition in $P$ by the direct sum of vector spaces and multiplication by the tensor product of vector spaces. The constant term of $P$ has to be interpreted as the appropriate direct sum of copies of $k$. Note that if $V$ is a representation of some proalgebraic group $G$, then $P(V,V^\vee)$ is also naturally a representation of $G$. (The constant term has to be interpreted as a trivial representation.)
The choice of a basis $v=(v_1,\ldots,v_n)$ of $V$ determines, for every $P\in \mathbb{N}[X,Y]$, a basis of $P(V,V^\vee)$, which we will call the \emph{$v$-canonical basis} of $P(V,V^\vee)$. It can be defined recursively as follows: 
\begin{itemize}
	\item The $v$-canonical basis of $V$ is $v$.
	\item The $v$-canonical basis of $V^\vee$ is the basis $v^\vee=(v_1^\vee,\ldots,v_n^\vee)$ dual to $v$.
	
	\item If $w_1,\ldots,w_m$ is the $v$-canonical basis of $W$ and $w'_1,\ldots,w'_{m'}$ is the $v$-canonical basis of $W'$, then $w_1,\ldots,w_m,w'_1,\ldots,w'_{m'}$ is the $v$-canonical basis of $W\oplus W'$.
	\item If $w_1,\ldots,w_m$ is the $v$-canonical basis of $W$ and $w'_1,\ldots,w'_{m'}$ is the $v$-canonical basis of $W'$, then $(w_i\otimes w'_j)_{1\leq i\leq m, 1\leq j\leq m'}$ is the $v$-canonical basis of $W\otimes_k W'$.
\end{itemize}	

Let $V$ be a not necessarily finite dimensional representation of a proalgebraic group $G$. For a $k$-algebra $R$, a subspace $W$ of $V$ is \emph{stable} under $g\in G(R)$ (or $g$ \emph{stabilizes} $W$) if $g\colon V\otimes_k R\to V\otimes_k R$ maps $W\otimes_k R$ into $W\otimes_k R$.  If $W$ is stable under all $g\in G(R)$ for all $R$, then $W$ is \emph{$G$-stable}. The subgroup of $G$ consisting of all $g$ that stabilize $W$ is a closed subgroup of $G$ called the \emph{stabilizer} of $W$.

For $n\geq 1$ and $d\geq 0$ we define 
$$P_d=\sum_{a+b\leq d}{ a+n-1 \choose a} { b+n-1 \choose b}X^aY^b\in \mathbb{N}[X,Y],$$
where the sum is taken over all pairs $(a,b)\in\mathbb{N}^2$ with $a+b\leq d$. The significance of this polynomial is explained in the following proposition.

\begin{prop} \label{prop: equivalence for degree}
	Let $V$ be a finite dimensional $k$-vector space and let $G$ be a closed subgroup of $\Gl_V$. For $d\geq 0$ the following closed subgroups of $\Gl_V$ are equal:
	\begin{enumerate}
		\item The subgroup of $\Gl_V$ defined by the ideal of $k[\Gl_V]$ generated by $\I(G)\cap k[\Gl_V]_{\leq d}$ (cf. Lemma~\ref{lemma: is Hopf ideal}).
		\item The subgroup of $\Gl_V$ consisting of all elements that stabilize all $G$-stable subspaces of $P(V,V^\vee)$ for all $P\in \mathbb{N}[X,Y]$ of degree at most $d$. 
		\item The subgroup of $\Gl_V$ consisting of all elements that stabilize all $G$-stable subspaces of $P_d(V,V^\vee)$.
	\end{enumerate}
Moreover, there exists a $G$-stable subspace $W$ of $P_d(V,V^\vee)$ such that the stabilizer of $W$ (in $\Gl_V$), agrees with the group defined in (i), (ii) and (iii).
\end{prop}
\begin{proof}
	For $j=1,2,3$ let $G_j$ denote the group defined in point $j$ above. Clearly $G_2\leq G_3$. Fixing a basis $v=(v_1,\ldots,v_n)$ of $V$, we may identify $\Gl_V$ with $\Gl_n$.

	To show that $G_1\leq G_2$ let $P\in \mathbb{N}[X,Y]$ be a polynomial of degree at most $d$ and let $W$ be a subspace of $P(V,V^\vee)$. Let $u_1,\ldots,u_r$ denote the $v$-canonical basis of $P(V,V^\vee)$. Then there exist $P_{ij}\in k[Z,Z^{-1}]_{\leq d}$ such that $g(u_i)=\sum_{j=1}^rP_{ij}(g)u_j$ for all $g\in \Gl_n(R)$ and all $k$-algebras $R$. It follows that for any basis $w_1,\ldots,w_r$ of $P(V,V^\vee)$ there exist $Q_{ij}\in k[Z,Z^{-1}]_{\leq d}$ such that $g(w_i)=\sum_{j=1}^r Q_{ij}(g)w_j$ for all $g\in \Gl_n(R)$ and all $k$-algebras $R$. We may extend a basis $w_1,\ldots,w_s$ of $W$ to a basis $w_1,\ldots,w_s,w_{s+1},\ldots, w_r$ of $P(V,V^\vee)$. Then, using the above notation, an element $g\in\Gl_n(R)$ stabilizes $W$ if and only if $Q_{ij}(g)=0$ for all $i$ and $j>s$. Thus, an element $g\in \Gl_n(R)$ such that $Q(g)=0$ for all $Q\in \I(G)\cap k[Z,Z^{-1}]_{\leq d}$, will stabilize every subspace $W$ of $P(V,V^\vee)$ stabilized by $G$. Hence $G_1\leq G_2$.
	
	The most difficult part now is to show that $G_3\leq G_1$. This follows from a detailed analysis of the proofs of two basic theorems on representations of algebraic groups (Theorems 4.14 and 4.27 in \cite{Milne:AlgebraicGroupsTheTheoryOfGroupSchemesOfFiniteTypeOverAField}): We consider the regular representation of $\Gl_n$ on $k[Z,Z^{-1}]$. This can be defined as the not necessarily finite dimensional representation of $\Gl_n$ corresponding to the comodule $k[Z,Z^{-1}]$ with comodule map 	
	$\Delta\colon k[Z,Z^{-1}]\to k[Z,Z^{-1}]\otimes_k k[\Gl_n],\ Z\mapsto Z\otimes Z$. Explicitly, the action of an element $g\in\Gl_n(R)$ on $k[Z,Z^{-1}]\otimes_k R$ is given by $g(f(Z))=f(Zg)$ for $f\in k[Z,Z^{-1}]$. We note that $k[Z,Z^{-1}]_{\leq d}$ is a $\Gl_n$-stable subspace of $k[Z,Z^{-1}]$. For $i=1,\ldots,n$ let $V_i$ denote the $k$-subspace of $k[Z,Z^{-1}]$ generated by the $i$-th row of $Z$. Then $V_i$ is $\Gl_n$-stable. In fact, $V_i$ is isomorphic to $V$ as a $\Gl_n$-representation, under the isomorphism that identifies $Z_{ij}$ with $v_j$ for $j=1,\ldots,n$. Similarly, for $j=1,\ldots,n$ let $W_j$ denote the subspace of $k[Z,Z^{-1}]$ generated by the $j$-th column of $Z^{-1}$. Then $W_j$ is $\Gl_n$-stable and indeed is isomorphic to $V^\vee$ as a $\Gl_n$-representation, under the isomorphism that identifies $(Z^{-1})_{ij}$ with $v^\vee_i$ for $i=1,\ldots,n$. Let $$Q=\sum_{d_1+\ldots+d_n+e_1+\ldots+e_n\leq d}X_1^{d_1}\ldots X_n^{d_n}Y_1^{e_1}\ldots Y_n^{e_n}\in\mathbb{N}[X_1,\ldots,X_n,Y_1,\ldots,Y_n]$$
	be the full polynomial of degree $d$ all of whose coefficients are equal to $1$. We have a surjective map $Q(V_1,\ldots,V_n,W_1,\ldots,W_n)\to k[Z,Z^{-1}]_{\leq d}$ of $\Gl_n$-representations, where $Q(V_1,\ldots,V_n,W_1,\ldots,W_n)$ is defined in a fashion similar to the definition of $P(V,V^\vee)$ above. Since $V_i\simeq V$ and $W_j\simeq V^\vee$ this can be interpreted as a surjective map $\pi\colon P_d(V,V^\vee)\to k[Z,Z^{-1}]_{\leq d}$ of $\Gl_n$-representations, where $P_d=Q(X,\ldots,X,Y,\ldots,Y)=\sum_{a+b\leq d}{ a+n-1 \choose a} { b+n-1 \choose b}X^aY^b\in \mathbb{N}[X,Y]$. 
	
	Since $\I(G)\cap k[Z,Z^{-1}]_{\leq d}$ is a $G$-stable subspace of $k[Z,Z^{-1}]_{\leq d}$, it follows that
	$W=\pi^{-1}(\I(G)\cap k[Z,Z^{-1}]_{\leq d})$ is a $G$-stable subspace of $P_d(V,V^\vee)$. Thus, by the definition of $G_3$, the subspace $W$ is $G_3$-stable. Therefore $\pi(W)= \I(G)\cap k[Z,Z^{-1}]_{\leq d}$ is $G_3$-stable. From this we deduce that $(\I(G)\cap k[Z,Z^{-1}]_{\leq d})=\I(G_1)$ is $G_3$-stable. However, only elements of $G_1$ stabilize $\I(G_1)$ (cf. \cite{Springer:LinearAlgebraicGroups} Lemma~ 2.3.8). Therefore $G_3\leq G_1$ as desired.
	
	Finally, the $G$-stable subspace $W=\pi^{-1}(\I(G)\cap k[Z,Z^{-1}]_{\leq d})$ of $P_d(V,V^\vee)$ has the property required in the last statement of the proposition.
\end{proof}

Let $V$ be a finite dimensional $k$-vector space and let $G$ be a closed subgroup of $\Gl_V$. For every $d\geq 0$ let $$G_{\leq d}$$ denote the closed subgroup of $\Gl_V$ characterized in Proposition \ref{prop: equivalence for degree}. We then have a descending chain
$$\Gl_V=G_{\leq 0}\supseteq G_{\leq 1}\supseteq G_{\leq 2}\supseteq \ldots$$
of closed subgroups of $\Gl_V$ that eventually stabilizes at $G$.

\begin{defi}
	The \emph{defining degree of $G$} is the smallest $d$ such that $G=G_{\leq d}$. 
\end{defi}

The following lemma will be used in the next subsection. Roughly speaking, it shows that stabilizers of subspaces of $P(V,V^\vee)$ are uniformly definable.

\begin{lemma} \label{lemma: define Qs}
	Let $P\in\mathbb{N}[X,Y]$ be a polynomial of degree $d$, $n\geq 1$ and let $s$ denote the dimension of the vector space $P(V,V^\vee)$, where $V$ is an $n$-dimensional vector space. Furthermore let $1\leq r\leq s$ and $1\leq i_1<\ldots<i_r\leq s$. Then there exist polynomials $Q_1,\ldots,Q_{(s-r)r}\in\mathbb{Z}[T,1/\det((T_{i_\ell,j})_{1\leq\ell,j\leq r}),Z,Z^{-1}]$ of degree at most $d$ in $Z$ and $Z^{-1}$, where $T=(T_{i,j})_{1\leq i\leq s, 1\leq j\leq r}$ and $Z=(Z_{i,j})_{1\leq i,j\leq n}$ such that the following holds: For every field $k$, every $k$-vector space $V$ of dimension $n$ with basis $\underline{v}=(v_1,\ldots,v_n)$, all matrices $A\in k^{s\times r}$ such that $\det(A_{i_\ell,j})_{1\leq\ell,j\leq r}\neq 0$, all $k$-algebras $R$ and all $g\in\Gl_V(R)\simeq \Gl_n(R)$ (via $\underline{v}$) the $k$-subspace of $P(V,V^\vee)$ generated by $\underline{u}A$, where $\underline{u}$ is the $\underline{v}$-canonical basis of $P(V,V^\vee)$, is stable under $g$ if and only if $Q_i(A,g)=0$ for $i=1,\ldots, (s-r)r$.
\end{lemma}
\begin{proof}
	Let $S=\mathbb{Z}[T,1/\det((T_{i_\ell,j})_{1\leq\ell,j\leq r})]$ and consider a free $S$-module $V_S$ of rank $n$ with basis $\underline{v}_S$. The definition of $P(V,V^\vee)$ and the $\underline{v}$-canonical bases of $P(V,V^\vee)$ extends from vector spaces to finite rank free modules in a straight forward manner.
	So let $\underline{u}_S$ denote the $\underline{v}_S$-canonical basis of $P(V_S,V_S^\vee)$. We extend the matrix $T\in S^{s\times r}$ to a matrix $\widetilde{T}\in\Gl_s(S)$ by adding the standard basis vectors $e_1,\ldots,e_{s-r}$ of length $s-r$ in the rows $\{1,\ldots,s\}\smallsetminus\{i_1,\ldots,i_r\}$. The group scheme $\Gl_{n,S}=\spec(S[Z,Z^{-1}])$ over $S$, acts linearly on $V_S$ and on $P(V_S,V_S^\vee)$. Moreover, there exist a matrix $B\in \mathbb{Z}[Z,Z^{-1}]^{s\times s}$ with entries of at most degree $d$ such that $g(\underline{u}_S)=\underline{u}_SB(g)$ for any $S$-algebra $S'$ and $g\in\Gl_n(S')$. It follows that
	$$g(\underline{u}_S\widetilde{T})=\underline{u}_SB(g)\widetilde{T}=\underline{u}_S\widetilde{T}(\widetilde{T}^{-1}B(g)\widetilde{T}).$$
	Thus the submodule of $V_S$ with basis $\underline{u}_ST$ is stable under $g$ if and only if the $(s-r)\times r$ submatrix in the lower left corner of $(\widetilde{T}^{-1}B(g)\widetilde{T})\in S[Z,Z^{-1}]^{s\times s}$ is zero.
	
	 We claim that the entries $Q_1,\ldots,Q_{(s-r)r}$ of this matrix have the required property. Since the entries of $B$ have degree at most $d$ in $Z$ and $Z^{-1}$, also $Q_1,\ldots,Q_{(s-r)r}$ have degree at most $d$ in $Z$ and $Z^{-1}$. The choice of a field $k$ and a matrix $A\in k^{s\times r}$ with $\det(A_{i_\ell,j})_{1\leq\ell,j\leq r}\neq 0$ defines a morphism of rings $S\to k$, i.e., a $k$-valued point of $\spec(S)$. The claim now follows by considering the fiber over this point. In more detail: The $k$-vector space $V_S\otimes_S k$ has basis $\underline{v}_S\otimes 1$ and we can define an isomorphism $V\to V_S\otimes_S k$ by  $\underline{v}\mapsto\underline{v}_S\otimes 1$. Similarly, we have an isomorphism $P(V,V^\vee)\to P(V_S,V_S^\vee)\otimes_S k,\ \underline{u}\mapsto \underline{u}_S\otimes 1$.
	We extend $A$ to a matrix $\widetilde{A}$ in $\Gl_s(k)$ is a similar fashion as we did with $T$. Then, for a $k$-algebra $R$ and $g\in\Gl_n(R)$ we have $g(\underline{u}\widetilde{A})=\underline{u}\widetilde{A}(\widetilde{A}^{-1}B(g)\widetilde{A})$. Thus the subspace of $P(V,V^\vee)$ generated by $\underline{u}A$ is stable under $g$ if and only if $Q_i(A,g)=0$ for $i=1,\ldots, (s-r)r$.
\end{proof}

\subsection{Types and stable subspaces}

Let $\M$ be a model of $\PROALG$ and let $(k,C,\omega)$ be the associated object of $\TANN$. For an object $b$ of $C$ we let $$G(b)\leq \Gl_{\omega(b)}$$ denote the (scheme-theoretic) image of the representation $G(\M)\to\Gl_{\omega(b)}$ defined by $b$. Note that an algebraic group is a quotient of $G(\M)$ if and only if it is isomorphic to some $G(b)$. Moreover, $G(\M)$ is the projective limit of the $G(b)$'s.
Our main result is that $\tp(b/k)$ determines $G(b)$. We begin by translating the main findings of the previous subsection into a statement about models of $\PROALG$. The point of the following corollary is that this somewhat clumsy characterization of when $G(b)_{\leq d}=G(b)_{\leq d'}$ can be expressed by an $\mathcal{L}$-formula.

\begin{cor} \label{cor: algebraic version of long formula}
	Let $\M$ be a model of $\PROALG$ and $(k,C,\omega)$ the associated object of $\TANN$. Let $b$ be an object of $C$ and for $0\leq d<d'$ let $s$ and $s'$ denote the dimensions of $P_d(\omega(b),\omega(b)^\vee)$ and $P_{d'}(\omega(b),\omega(b)^\vee)$ respectively. Then $G(b)_{\leq d}=G(b)_{\leq d'}$ if and only if the following condition is satisfied: For all bases $\underline{v}$ of $\omega(b)$, for all $r'$ with $1\leq r'\leq s'$, all $1\leq i'_1<\ldots<i'_{r'}\leq s$ and all $A'\in k^{s'\times r'}$ with $\det((A'_{i'_\ell,j})_{1\leq\ell,j\leq r'})\neq 0$ such that the subspace of $P_{d'}(\omega(b),\omega(b)^\vee)$ generated by $\underline{u}'A'$ is $G(\M)$-stable, where $\underline{u}'$ is the $\underline{v}$-canonical basis of $P_{d'}(V,V^\vee)$, there exist $1\leq r\leq s$, $1\leq i_1<\ldots<i_r\leq s$ and $A\in k^{s\times r}$ with $\det((A_{i_\ell,j})_{1\leq\ell,j\leq r})\neq 0$ such that the subspace of $P_d(V,V^\vee)$ generated by $\underline{u}A$ is $G(\M)$-stable, where $\underline{u}$ denotes the $\underline{v}$-canonical basis of $P_d(V,V^\vee)$, and the polynomials $Q'_1(A', Z),\ldots,Q'_{(s'-r')r'}(A',Z)\in k[Z,Z^{-1}]_{\leq d'}$ lie in the ideal of $k[Z,Z^{-1}]$ generated by $Q_1(A,Z),\ldots,Q_{(s-r)r}(A,Z)\in k[Z,Z^{-1}]_{\leq d}$. Here the polynomials $Q_1',\ldots,Q'_{(s'-r')r'}$ and $Q_1,\ldots,Q_{(s-r)r}$ are defined as in Lemma \ref{lemma: define Qs}.
\end{cor}
\begin{proof}
	Note that a subspace of some $P(\omega(b),\omega(b)^\vee)$ is $G(\M)$-stable if and only if it is $G(b)$-stable. Thus, according to Proposition \ref{prop: equivalence for degree} the closed subgroup $G(b)_{\leq d'}$ of $\Gl_{\omega(b)}$ is the stabilizer of some $G(\M)$-stable subspace $W'$ of $P_{d'}(V,V^\vee)$. If $\underline{v}$ is a basis of $\omega(b)$ and $\underline{u}'$ is the $\underline{v}$-canonical basis of $P_{d'}(V,V^\vee)$, then $W$ has a basis of the form $\underline{u}'A'$, for some $A'\in k^{s'\times r'}$ with linearly independent columns, i.e., $\det((A'_{i'_\ell,j})_{1\leq\ell,j\leq r'})\neq 0$ for an appropriate choice of $1\leq i'_1<\ldots<i'_{r'}\leq s$. So the defining ideal $\I(G(b)_{\leq d'})$ of $G(b)_{\leq d'}$ is generated by $Q'_1(A', Z),\ldots,Q'_{(s'-r')r'}(A',Z)$ according to Lemma~\ref{lemma: define Qs}. By construction, the polynomials $Q_1(A,Z),\ldots,Q_{(s-r)r}(A,Z)$ lie in the ideal $\I(G(b)_{\leq d})$. Thus, if $Q'_1(A', Z),\ldots,Q'_{(s'-r')r'}(A',Z)$ lie in the ideal generated by $Q_1(A,Z),\ldots,Q_{(s-r)r}(A,Z)$, then $\I(G(b)_{\leq d'})\subseteq \I(G(b)_{\leq d})$. So $G(b)_{\leq d}=G(b)_{\leq d'}$ as desired.
	
	Conversely, assume that $G(b)_{\leq d}=G(b)_{\leq d'}$. Similarly to the above argument, there exists  appropriate integers, $r,i_1,\ldots,i_r$ and a matrix $A\in k^{s\times r}$ such that $\det((A_{i_\ell,j})_{1\leq\ell,j\leq r})\neq 0$ and $\I(G(b)_{\leq d})$ is generated by $Q_1(A,Z),\ldots,Q_{(s-r)r}(A,Z)$. Since $\I(G(b)_{\leq d'})=\I(G(b)_{\leq d})$ it follows that $Q'_1(A', Z),\ldots,Q'_{(s'-r')r'}(A',Z)$ lie in the ideal generated by $Q_1(A,Z),\ldots,Q_{(s-r)r}(A,Z)$.
\end{proof}

Two verify that the above statement can be expressed by an $\mathcal{L}$-formula we need two lemmas.
Roughly speaking, the following lemma shows that
we can quantify over the $G(\M)$-stable subspaces of $P(\omega(b),\omega(b)^\vee)$.

\begin{lemma} \label{lemma: quantify over subspaces}
	Let $p=(m,n)$ and $P\in\mathbb{N}[X,Y]$. Let $s$ be the dimension of the vector space $P(V,V^\vee)$, where $V$ is an $n$-dimensional vector space. Then, for every $r$ with $1\leq r\leq s$, there exists an $\mathcal{L}$-formula $\varphi(T,y_1,\ldots,y_n)$, where $T=(T_{i,j})$ is an $s\times r$ matrix of variables from the field sort and $y_1,\ldots,y_n$ are variables from the $p$-total objects sort such that the following holds:
	
	For every model $\M=(k,B_p,X_p,B_{p,q},X_{p,q})$ of $\PROALG$, all $b\in B_p$, all bases $v_1,\ldots,v_n$ of $\omega(b)$ and all $A\in k^{s\times r}$ with linearly independent columns we have
	$\M\models \varphi(A,v_1,\ldots,v_n)$ if and only if the $k$-subspace of $P(\omega(b),\omega(b)^\vee)$ generated by $\underline{u}A$ is $G(\M)$-stable, where $\underline{u}$ is the $\underline{v}$-canonical basis of $P(\omega(b),\omega(b)^\vee)$.
\end{lemma}
\begin{proof}
	Let $\M=(k,B_p,X_p,B_{p,q},X_{p,q})$ be a model of $\PROALG$ and let $(k,C,\omega)$ be the associated object of $\TANN$. Moreover let $b\in B_p$ and let $P(b,b^\vee)\in B_{(1,s)}$ denote the unique tensor irreducible object of $C$ such that $\omega(P(b,b^\vee))\simeq P(\omega(b),\omega(b)^\vee)$ as $G(\M)$-representations. A subspace $W$ of $P(\omega(b),\omega(b)^\vee)$ of dimension $r$ is $G(\M)$-stable if and only if there exist $b'\in B_{(r,1)}$ and a morphism $h\colon b'\to P(b,b^\vee)$ in $C$ such that $\omega(h)\colon \omega(b')\to\omega(P(b,b^\vee))\simeq P(\omega(b),\omega(b)^\vee)$ has image $W$. So, in coordinates, if $\underline{v}$ is a basis of $\omega(b)$, and $A\in k^{s\times r}$ is such that $\underline{u}A$ is a basis of $W$, where $\underline{u}$ is the $\underline{v}$-canonical basis of $P(\omega(b),\omega(b)^\vee)$, then $W$ is $G(\M)$-stable if and only if there exists $b'\in B_{(r,1)}$ with a basis $\underline{v}'$ of $\omega(b')$ and a morphism $h\colon b'\to P(b,b^\vee)$ in $C$ such that $\omega(h)\colon \omega(b')\to\omega(P(b,b^\vee))\simeq P(\omega(b),\omega(b)^\vee)$ maps $\underline{v}'$ to $\underline{u}A$. Since the $\underline{v}$-canonical basis of $P(\omega(b),\omega(b)^\vee)$ can be characterized in terms of $\underline{v}$ by $\mathcal{L}$-formulas, the claim follows.
\end{proof}
%
We will also need a classical result of G. Hermann (cf. \cite[Theorem 3.4]{Aschenbrenner:IdealMembershipInPolynomialRingsOverTheIntegers}).
\begin{lemma}[G. Hermann \cite{Hermann:FrageDerEndlichVielenSchritte}] \label{lemma: Hermann}
	For every field $k$, if $f,g_1,\ldots,g_m\in k[X_1,\ldots,X_n]$ are polynomials of degree at most $d$ such that $f$ lies in the ideal generated by $g_1,\ldots,g_m$, then there exist polynomials $f_1,\ldots,f_m\in k[X_1,\ldots,X_n]$ of degree at most $(2d)^{2^n}$ such that $f=f_1g_1+\ldots+f_m g_m$. \qed
\end{lemma}	

Combining the above results we see that the set of all $b\in B_p$ such that $G(b)_{\leq d}=G(b)_{\leq d'}$ is definable:
 
\begin{prop} \label{prop: define defining degree}
For given $p$ and integers $0\leq d<d'$ there exists an $\mathcal{L}$-formula $\varphi(x)$ in one free variable $x$ belonging to the $p$-basic objects sort such that for every model $\M=(k,B_p,X_p,B_{p,q},X_{p,q})$ of $\PROALG$ and every $b\in B_p$ we have $\M\models \varphi(b)$ if and only if $G(b)_{\leq d}=G(b)_{\leq d'}$.
\end{prop}
\begin{proof}
	It suffices to see that the statement from Corollary \ref{cor: algebraic version of long formula} can be expressed by an $\mathcal{L}$-formula. This is guaranteed by Lemmas \ref{lemma: quantify over subspaces} and \ref{lemma: Hermann}.
\end{proof}

\begin{cor} \label{cor: type determines defining degree}
	Let $\M=(k,B_p,X_p,B_{p,q},X_{p,q})$ be a model of $\PROALG$. If $b,b'\in B_p$ are such that $\tp(b/\emptyset)=\tp(b'/\emptyset)$, then the defining degree of $G(b)$ agrees with the defining degree of $G(b')$.
\end{cor} 
	\begin{proof}
		Clear from Proposition \ref{prop: define defining degree}.
	\end{proof}
	
\begin{theo}
	Let $\M=(k,B_p,X_p,B_{p,q},X_{p,q})$ be a model of $\PROALG$ and let $b,b'\in B_p$. If $\tp(b/k)=\tp(b'/k)$, then there exists an isomorphism $\omega(b)\to\omega(b')$ of $k$-vector spaces that induces an isomorphism between $G(b)$ and $G(b')$.
\end{theo}
\begin{proof}
According to Corollary \ref{cor: type determines defining degree} the defining degree of $G(b)$ agrees with the defining degree of $G(b')$. Let us denote it with $d$. Let $s$ denote the dimension of $P_d(\omega(b),\omega(b)^\vee)$ and let $x$ be a variable from the $p$-objects sort. Let $1\leq r\leq s$ and $A\in k^{s\times r}$ with $\det((A_{i_\ell,j})_{1\leq\ell,j\leq r})\neq 0$ for $1\leq i_1<\ldots<i_r\leq s$. Moreover let $Q_1,\ldots,Q_{(s-r)r}\in\mathbb{Z}[T,1/\det((T_{i_\ell,j})_{1\leq\ell,j\leq r}),Z,Z^{-1}]$ be defined as in Lemma \ref{lemma: define Qs}.
Consider the formula $\varphi_A(x)$ with parameters from $k$ such that $\varphi_A(b)$ expresses the following statement: 

There exists a basis $\underline{v}$ of $\omega(b)$ such that for all $r'$ with $1\leq r'\leq s$, all $1\leq i'_1<\ldots<i'_{r'}\leq s$ and all $A'\in k^{s\times r'}$ with $\det((A'_{i'_\ell,j})_{1\leq\ell,j\leq r'})\neq 0$ such that the subspace of $P_d(\omega(b),\omega(b)^\vee)$ generated by $\underline{u}A'$ is $G(\M)$-stable, where $\underline{u}$ is the $\underline{v}$-canonical basis of $P_d(\omega(b),\omega(b)^\vee)$, the polynomials $Q'_1(A',Z),\ldots,Q'_{(s-r')r'}(A',Z)\in k[Z,Z^{-1}]_{\leq d}$ (with $Q'_1,\ldots,Q'_{(s-r')r'}$ defined as in Lemma \ref{lemma: define Qs}) lie in the ideal of $k[Z,Z^{-1}]$ generated by $Q_1(A,Z),\ldots,Q_{(s-r)r}(A,Z)\in k[Z,Z^{-1}]_{\leq d}$.

If $\varphi_A(x)$ lies in $\tp(b/k)$ the formula $\varphi_A(x)$ determines $G(b)$ because it shows that $\I(G(b))=\I(G(b)_{\leq d})$ is generated by $Q_1(A,Z),\ldots,Q_{(s-r)r}(A,Z)\in k[Z,Z^{-1}]_{\leq d}$. (Here $\Gl_{\omega(b)}\simeq \Gl_n$ via the basis $\underline{v}$ deemed to exist by $\varphi_A(x)$.)

On the other hand, $\varphi_A(x)$ lies in $\tp(b/k)$ for some choice of $r$, $A$ and $i_1,\ldots,i_r$.
\end{proof}

\bibliographystyle{alpha}
 \bibliography{bibdata}

\def\cprime{$'$}
\begin{thebibliography}{CvdDM82}

\bibitem[Asc04]{Aschenbrenner:IdealMembershipInPolynomialRingsOverTheIntegers}
Matthias Aschenbrenner.
\newblock Ideal membership in polynomial rings over the integers.
\newblock {\em J. Amer. Math. Soc.}, 17(2):407--441, 2004.

\bibitem[BHHW]{BachmayrHarbaterHartmannWibmer:FreeDifferentialGaloisGroups}
Annette Bachmayr, David Harbater, Julia Hartmann, and Michael Wibmer.
\newblock Free differential {G}alois groups.
\newblock arXiv:1904.07806.

\bibitem[Cha84]{Chatzidakis:ModelTheoryOfProfiniteGroupsPhDThesis}
Zoe~Maria Chatzidakis.
\newblock {\em M{ODEL} {THEORY} {OF} {PROFINITE} {GROUPS}}.
\newblock ProQuest LLC, Ann Arbor, MI, 1984.
\newblock Thesis (Ph.D.)--Yale University.

\bibitem[Cha87]{Chatzidakis:ModelTheoryOfProfiniteGroupsHavingIPIII}
Zo\'{e} Chatzidakis.
\newblock Model theory of profinite groups having {IP}. {III}.
\newblock In {\em Mathematical logic and theoretical computer science
  ({C}ollege {P}ark, {M}d., 1984--1985)}, volume 106 of {\em Lecture Notes in
  Pure and Appl. Math.}, pages 161--195. Dekker, New York, 1987.

\bibitem[Cha98]{Chatzidakis:ModelTheoryOfProfiniteGroupsHavingTheIwasawProperty}
Zo\'{e} Chatzidakis.
\newblock Model theory of profinite groups having the {I}wasawa property.
\newblock {\em Illinois J. Math.}, 42(1):70--96, 1998.

\bibitem[CvdDM81]{CherlinVanDenDriesMacintyre:DecidabilityAndUndecidabilityTheoremsforPACfields}
Gregory Cherlin, Lou van~den Dries, and Angus Macintyre.
\newblock Decidability and undecidability theorems for {PAC}-fields.
\newblock {\em Bull. Amer. Math. Soc. (N.S.)}, 4(1):101--104, 1981.

\bibitem[CvdDM82]{CherlinVanDenDriesMacinyre}
Gregory Cherlin, Lou van~den Dries, and Angus Macintyre.
\newblock The elementary theory of regularly closed fields.
\newblock {\em preprint}, 1982.

\bibitem[Del90]{Deligne:categoriestannakien}
Pierre Deligne.
\newblock Cat\'egories tannakiennes.
\newblock In {\em The {G}rothendieck {F}estschrift, {V}ol.\ {II}}, volume~87 of
  {\em Progr. Math.}, pages 111--195. Birkh\"auser Boston, Boston, MA, 1990.

\bibitem[DG70]{DemazureGabriel:GroupesAlgebriques}
Michel Demazure and Pierre Gabriel.
\newblock {\em Groupes alg\'ebriques. {T}ome {I}: {G}\'eom\'etrie alg\'ebrique,
  g\'en\'eralit\'es, groupes commutatifs}.
\newblock Masson \& Cie, \'Editeur, Paris; North-Holland Publishing Co.,
  Amsterdam, 1970.
\newblock Avec un appendice {{\i}t Corps de classes local} par Michiel
  Hazewinkel.

\bibitem[DM82]{DeligneMilne:TannakianCategories}
Pierre Deligne and James~S. Milne.
\newblock Tannakian categories.
\newblock In {\em Hodge cycles, motives, and {S}himura varieties}, volume 900
  of {\em Lecture Notes in Mathematics}, pages 101--228. Springer-Verlag,
  Berlin, 1982.

\bibitem[EF72]{EklofFischer:TheElementaryTheoryOfAbelianGroups}
Paul~C. Eklof and Edward~R. Fischer.
\newblock The elementary theory of abelian groups.
\newblock {\em Ann. Math. Logic}, 4:115--171, 1972.

\bibitem[EGNO15]{TensorCategories:EtingofGelakiNikshychOstrik}
Pavel Etingof, Shlomo Gelaki, Dmitri Nikshych, and Victor Ostrik.
\newblock {\em Tensor categories}, volume 205 of {\em Mathematical Surveys and
  Monographs}.
\newblock American Mathematical Society, Providence, RI, 2015.

\bibitem[Her26]{Hermann:FrageDerEndlichVielenSchritte}
Grete Hermann.
\newblock Die {F}rage der endlich vielen {S}chritte in der {T}heorie der
  {P}olynomideale.
\newblock {\em Math. Ann.}, 95(1):736--788, 1926.

\bibitem[Hod93]{Hodges:ModelTheory}
Wilfrid Hodges.
\newblock {\em Model theory}, volume~42 of {\em Encyclopedia of Mathematics and
  its Applications}.
\newblock Cambridge University Press, Cambridge, 1993.

\bibitem[Kam15]{Kamensky:ModelTheoryAndTheTannakianFormalism}
Moshe Kamensky.
\newblock Model theory and the {T}annakian formalism.
\newblock {\em Trans. Amer. Math. Soc.}, 367(2):1095--1120, 2015.

\bibitem[Mil17]{Milne:AlgebraicGroupsTheTheoryOfGroupSchemesOfFiniteTypeOverAField}
J.~S. Milne.
\newblock {\em Algebraic groups}, volume 170 of {\em Cambridge Studies in
  Advanced Mathematics}.
\newblock Cambridge University Press, Cambridge, 2017.
\newblock The theory of group schemes of finite type over a field.

\bibitem[Spr09]{Springer:LinearAlgebraicGroups}
T.~A. Springer.
\newblock {\em Linear algebraic groups}.
\newblock Modern Birkh\"auser Classics. Birkh\"auser Boston, Inc., Boston, MA,
  second edition, 2009.

\bibitem[vdPS03]{SingerPut:differential}
Marius van~der Put and Michael~F. Singer.
\newblock {\em Galois theory of linear differential equations}, volume 328 of
  {\em Grundlehren der Mathematischen Wissenschaften [Fundamental Principles of
  Mathematical Sciences]}.
\newblock Springer-Verlag, Berlin, 2003.

\bibitem[Vis06]{Visser:CategoriesOfTheoriesAndInterpretations}
Albert Visser.
\newblock Categories of theories and interpretations.
\newblock In {\em Logic in {T}ehran}, volume~26 of {\em Lect. Notes Log.},
  pages 284--341. Assoc. Symbol. Logic, La Jolla, CA, 2006.

\bibitem[Wat79]{Waterhouse:IntroductiontoAffineGroupSchemes}
William~C. Waterhouse.
\newblock {\em Introduction to affine group schemes}, volume~66 of {\em
  Graduate Texts in Mathematics}.
\newblock Springer-Verlag, New York, 1979.

\bibitem[Wib]{Wibmer:FreeProalgebraicGroups}
Michael Wibmer.
\newblock Free proalgebraic groups.
\newblock arXiv:1904.07455.

\end{thebibliography}

 \medskip
 
 \noindent Author information:

 \medskip
 
 \noindent Anand Pillay: Department of Mathematics, University of Notre Dame,  Notre Dame, IN  46556-4618, USA. Email: {\tt apillay@nd.edu}

 \medskip
 
 \noindent Michael Wibmer: Institute of Analysis of Number Theory, Graz University of Technology, 8010 Graz, Austria. Email: {\tt wibmer@math.tugraz.at}

\end{document}